\documentclass[12pt]{amsart}
\usepackage[dvips]{graphicx}
\usepackage{amsmath,graphics,xypic}
\usepackage{amsfonts,amssymb}
\usepackage[all]{xy}
\theoremstyle{plain}
\newtheorem{theorem*}{Theorem}
\newtheorem*{lemma*}{Lemma}
\newtheorem{corollary*}{Corollary}
\newtheorem*{proposition*}{Proposition}
\newtheorem{conjecture*}{Conjecture}
\newtheorem{theorem}{Theorem}[section]
\newtheorem{lemma}[theorem]{Lemma}
\newtheorem{corollary}[theorem]{Corollary}
\newtheorem{proposition}[theorem]{Proposition}

\newtheorem{conjecture}[theorem]{Conjecture}

\theoremstyle{remark}

\newtheorem*{remark}{Remark}
\newtheorem*{definition}{Definition}

\newtheorem*{claim}{Claim}

\newtheorem*{facta}{Fact 1}
\newtheorem*{factb}{Fact 2}
\newtheorem*{factc}{Fact 3}

\theoremstyle{definition}
\newtheorem{defn}[theorem]{Definition}

\def\stateh{\mathcal{H}}
\def\statefs{\mathcal{S}}

\def \CC{\mathcal{C}}
\def\fs{\mathcal{FS}}

\def\mm{\mathfrak{m}}
\def\zpp{\Z[\pi/\tipi]}

\def\gl{\mbox{GL}} \def\Q{\Bbb{Q}} \def\F{\Bbb{F}} \def\Z{\Bbb{Z}} \def\R{\Bbb{R}}  \def\C{\Bbb{C}}
\def\N{\Bbb{N}}   \def\ll{\langle} \def\rr{\rangle}
 \def\a{\alpha} \def\g{\gamma}  \def\bp{\begin{pmatrix}}
\def\sm{\setminus} \def\ep{\end{pmatrix}} \def\bn{\begin{enumerate}} 
 \def\rank{\mbox{rank}} \def\div{\mbox{div}} \def\en{\end{enumerate}}
\def\ba{\begin{array}} \def\ea{\end{array}}  
\def\intt{\mbox{int}} \def\S{\Sigma}  \def\a{\alpha} \def\b{\beta} \def\ti{\tilde}
\def\id{\mbox{id}}  \def\im{\mbox{Im}} 
  \def\ker{\mbox{Ker}}
\def\ker{\mbox{Ker}}\def\be{\begin{equation}} \def\ee{\end{equation}} 
   
 \def\hom{\mbox{Hom}}  
 \def\aut{\mbox{Aut}}  
 \def\dim{\mbox{dim}}

\def\zt{\Z[t^{\pm 1}]} \def\qt{\Q[t^{\pm 1}]}   
\def\w{\omega}   
   \def\rt{R[t^{\pm
1}]} \def\vt{V[t^{\pm
1}]} \def\fr12{\frac{1}{2}} \def\z12{\Z[\fr12]}

\def\fpt{\F_p[t^{\pm 1}]}
\def\fp{\F_p}

\def\tpm {[t^{\pm 1}]}

\def\i{\iota}

\def\ol{\overline}

\def\tipi{\ti{\pi}}
\def\hatpi{\hat{\pi}}

\def\G{\Gamma}

\def\deg{\mbox{deg}}

\begin{document}

\title{Twisted Alexander polynomials detect fibered 3--manifolds}
\author{Stefan Friedl}
\address{University of Warwick, Coventry, UK}
\email{s.k.friedl@warwick.ac.uk}
\author{Stefano Vidussi}
\address{Department of Mathematics, University of California,
Riverside, CA 92521, USA} \email{svidussi@math.ucr.edu}
\thanks{S. Friedl was  supported by a CRM--ISM Fellowship and by CIRGET}
\thanks{S. Vidussi was partially supported by a University of California Regents' Faculty Fellowships and by  NSF grant \#0906281.}

\date{\today}
\begin{abstract}
A classical result in knot theory says that for a fibered knot the Alexander polynomial  is monic and that the degree equals twice the genus of the knot.
This result has been generalized by various authors to twisted Alexander polynomials and fibered 3--manifolds.
In this paper we show that the conditions on twisted Alexander polynomials are not only necessary but also sufficient for a 3--manifold to be fibered.
By previous work of the authors this result implies that if a manifold of the form $S^1 \times N^3$
 admits a symplectic structure, then $N$ fibers over $S^1$. In fact we will completely determine  the
symplectic cone of $S^1\times N$ in terms of the fibered faces of the Thurston norm ball of $N$.
\end{abstract}
\maketitle

\section{Introduction}

\subsection{Twisted Alexander polynomials and fibered 3--manifolds}
Let $N$ be a compact, connected, oriented $3$--manifold with empty or toroidal boundary.
Given a nontrivial class $\phi \in H^1(N;\Z)=\hom(\pi_1(N),\Z)$
 we say that $(N,\phi)$ \textit{fibers over $S^{1}$} if there exists
 a fibration $f:N\to S^1$ such that the induced map $f_*:\pi_1(N)\to \pi_1(S^1)=\Z$ agrees with $\phi$.
 Stated otherwise, the homotopy class in $[N,S^1] = H^1(N;\Z)$ identified by $\phi$ can be represented by a fibration.

It is a classical result in knot theory that if a knot $K\subset S^3$ is fibered, then the Alexander polynomial is monic
 (i.e. the top coefficient equals $\pm 1$), and the degree of the Alexander polynomial equals twice the genus of the knot.
This result has been generalized in various directions by several authors
(e.g. \cite{McM02,Ch03,GKM05,FK06,Ki07}) to show that twisted Alexander polynomials
give necessary conditions for $(N,\phi)$ to fiber.

To formulate this kind of result more precisely
we have to introduce some definitions. Let $N$ be a 3--manifold with empty or toroidal boundary
  and let $\phi \in H^1(N;\Z)$.
 Given $(N,\phi)$ the \emph{Thurston norm} of $\phi$ (cf. \cite{Th86}) is defined as
 \[
||\phi||_{T}=\min \{ \chi_-(S)\, | \, S \subset N \mbox{ properly embedded surface dual to }\phi\}.
\]
Here, given a surface $S$ with connected components $S_1\cup\dots \cup S_k$, we define
$\chi_-(S)=\sum_{i=1}^k \max\{-\chi(S_i),0\}$.

In the following we assume that $\phi \in H^1(N;\Z)$ is non--trivial. Let $\a:\pi_1(N)\to G$ be a
homomorphism to a finite group.
We have the permutation representation $\pi_1(N) \to \aut(\Z[G])$ given by left multiplication, which we also denote by $\a$. We can therefore consider the twisted Alexander polynomial $\Delta_{N,\phi}^{\a}\in \zt$, whose definition is detailed in Section \ref{sectionufd}.
We denote by $\phi_{\a}$  the restriction of $\phi\in H^1(N;\Z)=\hom(\pi_1(N),\Z)$ to $\ker(\a)$.
Note that $\phi_{\a}$ is necessarily non--trivial.
We denote by  $\div \phi_{\a} \in \N$ the divisibility of $\phi_{\a}$, i.e.
\[ \div \phi_{\a} =\max\{ n\in \N \, |\, \phi_{\a}=n\psi \mbox{ for some }\psi:\ker(\a)\to \Z\}.\]
We can now formulate the following theorem which appears as  \cite[Theorem~1.3 and Remark~p.~938]{FK06}.

\begin{theorem} \label{friedlkim} Let  $N \neq S^1 \times S^2, S^1 \times D^2$ be a $3$--manifold with empty or toroidal boundary.
 Let $\phi \in H^{1}(N;\Z)$ a nontrivial class. If $(N,\phi)$ fibers over $S^1$, then for
any homomorphism $\a:\pi_1(N)\to G$ to a finite group
the twisted Alexander polynomial $\Delta_{N,\phi}^{\a}\in \zt$ is monic
and \[  \deg(\Delta_{N,\phi}^{\a})= |G| \, \|\phi\|_{T} + (1+b_3(N)) \div \phi_{\a}.\]
 \end{theorem}

 It is well known that in general the constraint of monicness and degree for the ordinary Alexander polynomial falls short
 from characterizing fibered 3--manifolds.
The main result of this paper is to show that on the other hand the collection of all twisted Alexander polynomials does detect fiberedness, i.e. the converse of Theorem \ref{friedlkim} holds true:

\begin{theorem} \label{mainthm} Let $N$ be a $3$--manifold with empty or toroidal boundary.
 Let $\phi \in H^{1}(N;\Z)$ a nontrivial class. If for
 any homomorphism $\a:\pi_1(N)\to G$ to a finite group
the twisted Alexander polynomial $\Delta_{N,\phi}^{\a}\in \zt$ is monic
and \[ \deg(\Delta_{N,\phi}^{\a})= |G| \, \|\phi\|_{T} + (1+b_3(N)) \div \phi_{\a}\] holds, then $(N,\phi)$ fibers over $S^1$. \end{theorem}

Note that alternatively it is possible to rephrase this statement in terms of Alexander polynomials of the finite regular covers of $N$, using the fact that $\Delta_{N,\phi}^{\a} = \Delta_{{\tilde N},p^*(\phi)}$
(cf. \cite{FV08a}), where $p: {\tilde N} \to N$ is the cover of $N$ determined by $\ker(\a)$.

Note that this theorem asserts that twisted Alexander polynomials detect whether $(N,\phi)$ fibers under the assumption that $||\phi||_T$ is known;
while it is known that twisted Alexander polynomials give lower bounds (cf. \cite[Theorem~1.1]{FK06}),
it is still an open question whether twisted Alexander polynomials determine the Thurston norm.\


In the case where $\phi$ has trivial Thurston norm, this result is proven in \cite{FV08b}, using subgroup separability.
Here, following a different route (see Section \ref{strat} for a summary of the proof),
we prove the general case.

\subsection{Symplectic $4$--manifolds and twisted Alexander polynomials}
In 1976 Thurston  \cite{Th76} showed that if a closed $3$--manifold $N$ admits a
fibration over $S^1$, then $S^1 \times N$ admits a symplectic structure, i.e. a closed, nondegenerate $2$--form $\omega$.
It is natural to ask whether the converse to this statement holds true. In its simplest form, we can state this problem in the following way:
\begin{conjecture} \label{taubes} Let $N$ be a closed $3$--manifold. If $S^1 \times N$ is symplectic, then there exists
a $\phi \in H^1(N;\Z)$ such that $(N,\phi)$ fibers over $S^1$.\end{conjecture}

Interest in this question was motivated by Taubes' results in the study of Seiberg-Witten
invariants of symplectic $4$--manifolds (see \cite{Ta94,Ta95}), that gave initial evidence to an
affirmative solution of this conjecture. In the special case where $N$ is obtained via $0$--surgery
along a knot in $S^3$, this question appears also in \cite[Question~7.11]{Kr98}.
Over the last ten years evidence for this conjecture was given by various authors
\cite{Kr98,CM00,Et01,McC01,Vi03}.

In \cite{FV08a} the authors initiated a project relating Conjecture \ref{taubes} to the  study of twisted Alexander polynomials. The outcome of that
investigation is that if $S^1 \times N$ is symplectic, then the twisted Alexander polynomials
of $N$ behave like twisted Alexander polynomials of a fibered $3$--manifold. More precisely, the following holds (cf. \cite[Theorem~4.4]{FV08a}):

\begin{theorem} \label{thm:fv06}
Let $N$ be an irreducible closed $3$--manifold and $\w$ a symplectic structure on $S^1\times N$ such that $\w$ represents an integral cohomology class.
Let $\phi \in H^{1}(N;\Z)$ be the K\"unneth component of $[\w]\in H^2(S^1\times N;\Z)$.
Then for any homomorphism $\a:\pi_1(N)\to G$ to a finite group
the twisted Alexander polynomial $\Delta_{N,\phi}^{\a}\in \zt$ is monic
and \[  \deg(\Delta_{N,\phi}^{\a})= |G| \, \|\phi\|_{T} + 2\div \phi_{\a}.\]
 \end{theorem}

Note that it follows from McCarthy's work \cite{McC01} (see also Lemma \ref{lem:prime}) and Perelman's proof of the geometrization conjecture (cf. e.g. \cite{MT07})
that   if $S^1 \times N$ is symplectic, then $N$ is prime,
i.e. either irreducible or $S^1 \times S^2$.
The proof of Theorem \ref{thm:fv06} relies heavily on the results of \cite{Kr98} and \cite{Vi03},
which in turn build on results of Taubes \cite{Ta94,Ta95} and Donaldson \cite{Do96}.

As the symplectic condition is open, the assumption that a symplectic manifold admits an integral symplectic form is not restrictive.
Therefore, combining Theorem \ref{mainthm} with Theorem \ref{thm:fv06}, we deduce that Conjecture \ref{taubes} holds true.
In fact, in light of  \cite[Theorems~7.1~and~7.2]{FV07}, we have the following more refined statement:

\begin{theorem}\label{thm:mainresult}
Let $N$ be a closed oriented 3--manifold. Then given $\Omega \in H^2(S^1\times N;\R)$ the following are equivalent:
\bn
\item $\Omega$ can be represented by a symplectic structure;
\item $\Omega$ can be represented by a symplectic structure which is $S^1$--invariant;
\item $\Omega^2>0$ and the K\"unneth component $\phi \in H^1(N;\R)$ of $\Omega$ lies in the open cone on a fibered face of the Thurston norm ball of $N$.
\en
\end{theorem}

Note that the theorem allows us in particular to completely determine the symplectic cone of a manifold of the form $S^1\times N$ in terms of the fibered cones of $N$.

Combined with the results of \cite{FV07, FV08a}, Theorem \ref{mainthm} shows in particular that the collection of the
Seiberg-Witten invariants of all finite covers of $S^1\times N$ determines whether
$S^1\times N$ is symplectic or not.
In particular, we have the following corollary (we refer to \cite{Vi99,Vi03} for the notation and the formulation in the case that $b_1(N)=1$).

\begin{corollary}
Let $N$ be a closed 3--manifold with $b_1(N)>1$. Then given a spin$^c$ structure $K \in H^2(S^1 \times N;\Z)$ there
exists a symplectic structure representing a cohomology class $\Omega \in H^2(S^1\times N; \R)$ with canonical class $K$
if and only if the following conditions hold:
\bn
\item $K\cdot \phi = \|\phi \|_{T}$, where $\phi \in H^1(N;\R)$ is the K\"unneth component of $\Omega$,
\en
and  for any regular finite cover $p:{\tilde N}\to N$
\bn
\item[(2)] $SW_{S^1\times {\tilde N}}(p^*(K))=1$,
\item[(3)] for any Seiberg--Witten basic class $\kappa \in H^2(S^1\times {\tilde N};\Z)$ we have
\[ |p_*(\kappa) \cdot \phi| \leq \mbox{deg}(p)~ K\cdot \phi,  \]
(where $p_*$ is the transfer map)
and the latter equality holds if and only if $\kappa=\pm p^{*}K$.
\en
\end{corollary}

(Note that, under the hypotheses of the Corollary, all basic classes of $S^1 \times {\tilde N}$ are the pull--back of elements of $H^2({\tilde N};\Z)$.)\\

\noindent \textbf{Remark.}
A different approach to Conjecture \ref{taubes} involves a deeper investigation of
the consequence of the symplectic condition on $S^1 \times N$, that goes beyond the information
 encoded in Theorem \ref{thm:fv06}. A major breakthrough in this direction has recently been obtained
 by Kutluhan and Taubes (\cite{KT09}). They show that
  if $N$ is a 3--manifold  such that $S^1\times N$ is symplectic, under some cohomological assumption on the symplectic form,  then the Monopole Floer homology of $N$ behaves like the
 Monopole Floer homology of a fibered 3--manifold.
On the other hand it is known, due to the work of Ghiggini, Kronheimer and Mrowka, and Ni
that Monopole Floer homology detects fibered 3--manifolds
(\cite{Gh08,Ni09,KM08,Ni08}).
 The combination of  the above results proves in particular Conjecture \ref{taubes}
in the case that $b_1(N)=1$.


\subsection{Fibered 3--manifolds and finite solvable groups: outline of the proof}\label{strat}

In this subsection we will outline the strategy of the proof of Theorem \ref{mainthm}.
It is useful to introduce the following definition.
\begin{definition} Let $N$ be a $3$--manifold
with empty or toroidal boundary, and let $\phi \in H^1(N;\Z)$ be a nontrivial class. We say that $(N,\phi)$ satisfies Condition ($*$) if for
any homomorphism $\a:\pi_1(N)\to G$ to a finite group
the twisted Alexander polynomial $\Delta_{N,\phi}^{\a}\in \zt$ is monic
and \[  \deg(\Delta_{N,\phi}^{\a})= |G| \, \|\phi\|_{T} + (1+b_3(N)) \div \phi_{\a}.\]
\end{definition}

It is well--known (see \cite{McC01} for the closed case, and \ref{lem:prime} for the general case) that Condition ($*$) implies, using geometrization, that $N$ is prime, so we can restrict ourself to the case where $N$ is irreducible.

Note that McMullen \cite{McM02} showed that, when the class $\phi$ is primitive, the condition $\Delta_{N,\phi}\ne 0$ implies that there
exists a connected Thurston norm minimizing surface
$\S$ dual to $\phi$. It is well--known that to prove Theorem \ref{mainthm} it is sufficient to consider a primitive $\phi$,
 and we will assume that in the following. Denote $M = N\sm \nu \S$; the boundary of $M$ contains two copies $\S^{\pm}$ of $\S$
and throughout the paper we denote the inclusion induced maps $\S \to \S^\pm \to M$ by $\i_{\pm}$.

By Stallings' theorem \cite{St62} the surface $\S$ is a fiber of a fibration
$N\to S^1$ if and only if $\i_{\pm}:\pi_1(\S) \to \pi_1(M)$ are isomorphisms.
Hence to prove Theorem \ref{mainthm} we need to show that if $(N,\phi)$ satisfies Condition ($*$), then the
monomorphisms  $\i_{\pm}:\pi_1(\S) \to \pi_1(M)$ are in fact isomorphisms.
Using purely group theoretic arguments we are not able to show directly that Condition  ($*$)
implies the desired isomorphism; however, we have  the following result:

\begin{proposition}\label{thm:finitesolvable}
Assume that $(N,\phi)$ satisfies Condition  $(*)$ and that $\phi$ is primitive.
Let $\S\subset N$ be a connected Thurston norm minimizing surface dual to $\phi$ and let $\i$ be either of
the two inclusion maps of $\S$ into $M = N\sm \nu \S$. Then
$\i:\pi_1(\S)\to \pi_1(M)$ induces an isomorphism of the prosolvable completions.
\end{proposition}

We refer to Section \ref{section:pro} for information regarding group completions.
Proposition \ref{thm:finitesolvable} translates the information from Condition ($*$) into information regarding the maps $\i_\pm:\pi_1(\S)\to \pi_1(M)$.
From a purely group theoretic point of view it is a difficult problem to decide whether a homomorphism which gives
rise to an isomorphism of prosolvable completions has to be an isomorphism itself
(cf. \cite{Gr70}, \cite{BG04}, \cite{AHKS07} and also Lemma \ref{lem:isometab}). But in our  3--dimensional setting we can use a recent result of Agol \cite{Ag08} to prove the following theorem.

\begin{theorem} \label{thm:agolintro}
Let $N$ be an irreducible 3--manifold with empty or toroidal boundary.
Let $\S\subset N$ be a connected Thurston norm minimizing surface. We write $M=N\sm \nu \S$.
Assume the following hold:
\bn
\item  the inclusion induced maps $\i_\pm : \pi_1(\S)\to \pi_1(M)$
give rise to isomorphisms of the respective
prosolvable completions, and
\item  $\pi_1(M)$ is residually finite solvable,
\en
then $\i_\pm:\pi_1(\S)\to \pi_1(M)$ are isomorphisms, hence $M = \Sigma \times I$.
\end{theorem}

In light of Proposition \ref{thm:finitesolvable}, the remaining obstacle
for the proof of Theorem \ref{mainthm}
 is the condition in Theorem \ref{thm:agolintro}  that $\pi_1(M)$
has to be residually finite solvable.
It is well--known  that linear groups (and hence in particular hyperbolic 3--manifolds groups) are virtually residually $p$
for all but finitely many primes $p$ (cf. e.g. \cite[Theorem~4.7]{We73} or \cite[Window~7,~Proposition~9]{LS03}),
in particular they are residually finite solvable. Thurston conjectured that 3--manifold groups in general are
linear (cf. \cite[Problem~3.33]{Ki}), but this is still an open problem.
Using the recent proof of the geometrization conjecture (cf. e.g. \cite{MT07}) we will prove the following result, which will be enough for our purposes.

\begin{theorem}\label{thm:jsjresp}\label{thm:virtpintro}
Let $N$ be a closed
prime 3--manifold. Then for all but finitely many primes $p$ there
exists a finite cover $N'$ of $N$ such that the fundamental group of any component of the JSJ decomposition of
$N'$ is residually a $p$--group.
\end{theorem}

 We can now deduce Theorem \ref{mainthm} as follows: We first show in Lemmas
 \ref{lem:prime} and \ref{lem:closed} that it suffices to show Theorem \ref{mainthm} for closed prime 3--manifolds.
 Theorem \ref{mainthm} in that situation
 now follows from combining Theorems \ref{thm:finitesolvable}, \ref{thm:agolintro} and \ref{thm:jsjresp} with a more technical theorem which allows us to treat the various JSJ pieces separately
 (cf. Theorem \ref{thm:miiso}).
\\

Added in proof: In a very recent  paper (\cite{AF10}) Matthias Aschenbrenner and the first author showed that
any 3--manifold group is  virtually residually $p$. This simplifies the proof of Theorem \ref{mainthm} as outlined in \cite{FV10}. 
\\

This paper is structured as follows. In Section \ref{section:basics} we recall the definition of twisted Alexander polynomials and some basics regarding
completions of groups. In Section \ref{section:monic} we will prove Proposition \ref{thm:finitesolvable} and in Section  \ref{section:product} we give the proof of Theorem \ref{thm:agolintro}.
In Section \ref{section:virtp} we prove Theorem \ref{thm:jsjresp}
and in Section \ref{section:jsj} we provide the proof for Theorem \ref{thm:miiso}.
Finally in Section \ref{section:mainthm} we complete the proof of Theorem \ref{mainthm}.
\\

\noindent \textbf{Conventions and notations.}
 Throughout the paper, unless otherwise stated, we will assume that all
manifolds are oriented and connected, and all homology and cohomology groups have integer coefficients.
Furthermore all surfaces are assumed to be properly embedded and all spaces are compact and connected,
unless it says explicitly otherwise.
 The derived series of a group
$G$ is defined inductively by $G^{(0)}=G$ and
 $G^{(n+1)}=[G^{(n)},G^{(n)}]$.
\\

\noindent \textbf{Acknowledgments.}
We would like to thank
Ian Agol,
Matthias Aschenbrenner,
Steve Boyer,
Paolo Ghiggini,
Taehee Kim,
Marc Lackenby,
Alexander Lubotzky,
Kent Orr,
Saul Schleimer,
Jeremy van Horn--Morris and
Genevieve Walsh
for many helpful comments and conversations.
We also would like to thank the referee for suggesting several improvements to the paper and pointing out various inaccuracies.
\\

\section{Preliminaries: Twisted invariants and completions of groups}\label{section:basics}

\subsection{Twisted homology}

Let $X$ be  a CW--complex  with base point $x_0$.  Let $R$ be a commutative ring, $V$ a module over $R$ and
$\a:\pi_1(X,x_0)\to \aut_R(V)$  a representation.
Let $\ti{X}$ be the universal cover of $X$. Note that $\pi_{1}(X,x_0)$ acts on the left on $\ti{X}$ as
group of deck transformations. The cellular chain groups $C_*(\ti{X})$ are in a natural way right
$\pi_1(X)$--modules, with the right action on $C_{*}(\ti{X})$ defined via $\sigma \cdot g :=
g^{-1}\sigma$, for $\sigma \in C_{*}(\ti{X})$. We can form by tensoring the chain complex
$C_*(\ti{X})\otimes_{\Z[\pi_1(X,x_0)]}V$,
which is  a complex of  $R$--modules.
Now define $H_{i}(X;V):=
H_i(C_*(\ti{X})\otimes_{\Z[\pi_1(X,x_0)]}V)$.
The isomorphism type of the $R$--module $H_{i}(X;V)$  does not depend on the choice of the base point,
in fact it only  depends  on the homotopy type of  $X$ and the isomorphism type of the representation.

In this paper we will also frequently consider twisted homology for a finitely generated group $\Gamma$; its definition can be reduced to the one above by looking at the twisted homology of the Eilenberg-Maclane space $K(\Gamma,1)$.

The most common type of presentation we consider in this paper is as follows:
Let $X$ be a topological space,  $\a:\pi_1(X)\to G$ a homomorphism to a group $G$
and  $H\subset G$  a subgroup of finite index.
Then we get a natural action of $\pi_1(X)$
on $\aut_{\Z}(\Z[G/H])$ by left--multiplication, which gives rise to the homology groups $H_i(X;\Z[G/H])$.

We will now study the $\Z[\pi_1(X)]$--module $\Z[G/H]$ in more detail.
We write $C:=\a(\pi_1(X))$.
  Consider the set of double cosets $C\backslash G /H$. By definition  $g,g'\in G$ represent the same equivalence class if and only if there exist $c,c'\in C$ and $h,h'\in H$ such that
$c gh=c'g'h'$.
Note that $g_1,\dots,g_k\in G$ are a complete set of representatives of $C\backslash G /H$ if and only if $G$ is the disjoint union of $C g_1H,\dots,C g_kH$.
The first part of the following lemma is an immediate consequence of \cite[II.5.2.]{Br94}, the second part follows either from Shapiro's lemma
or a straightforward calculation.

\begin{lemma}\label{lem:group2}
Let  $g_1,\dots,g_k\in G$ be a  set of representatives for the equivalence classes $C\backslash G /H$.
For $i=1,\dots,k$ write  ${\tilde C}_i = C \cap g_i H g_i^{-1}$.
We then have the following isomorphisms of left $\Z[C]$--modules:
\[ \Z[G/H] \cong  \bigoplus_{i=1}^k \Z[C/{\tilde C}_i].\]
In particular $H_0(X;\Z[G/H])$ is a free abelian group of rank $k=|C\backslash G/H|$.
\end{lemma}

\subsection{Induced maps on low dimensional homology groups}
In this section we will give criteria when maps between groups give rise to isomorphisms between low dimensional twisted homology groups.
We start out with a study of the induced maps on 0--th twisted homology groups.

\begin{lemma} \label{lem:group1}
Let $\varphi :A\to B$ be a monomorphism of finitely generated groups.
Suppose that $B$ is a subgroup of a group $\pi$ and let $\tipi\subset \pi$ be a
 subgroup of finite index.
 Let  $g_1,\dots,g_k\in \pi$ be a  set of representatives for the equivalence classes $B\backslash \pi /\tipi$.
For $i=1,\dots,k$ we write $\ti{B}_i=B\cap g_i\ti{\pi}g_i^{-1}$ and $\ti{A}_i=\varphi^{-1}(\ti{B}_i)$.
Then
 \[ \varphi_{*}:H_0(A;\Z[{\pi/\ti{\pi}}])\to H_0(B;\Z[{\pi/\ti{\pi}}])\]
 is an epimorphism of free abelian groups and it is an isomorphism if and only if
$\varphi: A/\ti{A}_i \to B/\ti{B}_i$ is a bijection  for any $i$.
\end{lemma}

\begin{proof}
It is well--known that the induced map on 0--th twisted homology groups is always
surjective (cf. e.g. \cite[Section~6]{HS97}) and by Lemma \ref{lem:group2} both groups are free abelian groups.
Now note that without loss of generality we can assume that $A\subset B$ and that $\varphi$ is the inclusion map.
It follows from Lemma \ref{lem:group2} that $H_0(B;\zpp)$ is a free abelian group of rank $k=|B\backslash \pi/\tipi|$.
By the same Lemma we also have
\[ \Z[\pi/\tipi] \cong  \bigoplus_{i=1}^k \Z[B/{\tilde B}_i]\]
as left $\Z[B]$--modules  and hence also as left $\Z[A]$--modules.
By applying Lemma \ref{lem:group2} to the $\Z[A]$--modules $\Z[B/{\tilde B}_i]$
we see that $H_0(A;\zpp)$ is a free abelian group of rank $k$ if and only if
$|A\backslash B/\ti{B}_i|=1$  for any $i$.
It is straightforward to see that this is equivalent to $A/\ti{A}_i \to B/\ti{B}_i$ being a bijection  for any $i$.
\end{proof}

We will several times make use of the following corollary.

\begin{corollary}\label{cor:group1}
Let $\varphi :A\to B$ be a monomorphism of finitely generated groups.
Let $\b:B\to G$  be a homomorphism to a finite group.
Then
 \[ \varphi_{*}:H_0(A;\Z[G])\to H_0(B;\Z[G]) \] is an epimorphism of free abelian groups
 and it is an isomorphism
 if and only if
 \[ \im\{A\to B\to G\}=\im\{B\to G\}.\]
\end{corollary}

\begin{proof}
Let $\pi'=B\times G$ and $\tipi'=B$. We can then apply Lemma \ref{lem:group1}
to $A'=A, B'=\{ (g,\b(g)) \, |\, g\in B\}\subset \pi'$ and $\varphi'(a)=(\varphi(a),\b(\varphi(a)), a\in A'$.
It is straightforward to verify that the desired equivalence of statements follows.
\end{proof}

We now turn to the question when group homomorphisms induce isomorphisms of the 0--th and the first twisted homology groups at the same time.

\begin{lemma}\label{lem:metabany}
Let $\varphi :A\to B$ be a monomorphism of finitely generated groups.
Suppose that $B$ is a subgroup of a group $\pi$ and let $\tipi\subset \pi$ be a
 subgroup of finite index.
 Let  $g_1,\dots,g_k\in \pi$ be a  set of representatives for the equivalence classes $B\backslash \pi /\tipi$.
For $i=1,\dots,k$ we write $\ti{B}_i=B\cap g_i\ti{\pi}g_i^{-1}$ and $\ti{A}_i=\varphi^{-1}(\ti{B}_i)$.
Then
 \[ \varphi_{*}:H_i(A;\Z[{\pi/\ti{\pi}}])\to H_i(B;\Z[{\pi/\ti{\pi}}]) \] is an isomorphism for $i=0$ and $i=1$
 if and only if the following two conditions are satisfied:
 \bn
\item $\varphi: A/\ti{A}_i \to B/\ti{B}_i$ is a bijection for any $i$,
\item $\varphi:A/ [ \ti{A}_i,\ti{A}_i] \to B/[\ti{B}_i,\ti{B}_i]$ is a bijection for any $i$.
\en
\end{lemma}

\begin{proof}
Without loss of generality we can assume that $A\subset B$ and that $\varphi$ is the inclusion map.
By Lemmas \ref{lem:group2} and  \ref{lem:group1}
 it suffices to show for any $i$ the following:
 If $A/\ti{A}_i\to B/\ti{B}_i$ is a bijection,
 then the map
$ H_1(A;\Z[{B/\ti{B}_i}])\to H_1(B;\Z[{B/\ti{B}_i}])$ is an isomorphism if and only if $ \varphi:A/ [ \ti{A}_i,\ti{A}_i] \to B/[\ti{B}_i,\ti{B}_i]$ is a bijection.

Using the above and using Shapiro's Lemma we can identify
\[ \ba{rcl}H_1(A;\Z[B/\ti{B}_i])=H_1(A;\Z[A/\ti{A}_i])&=&\ti{A}_i/[\ti{A}_i,\ti{A}_i]\hspace{1cm} \mbox{and}\\
H_1(B;\Z[{B/\ti{B}_i}])&=&\ti{B}_i/[\ti{B}_i,\ti{B}_i].\ea \]
Note that $A/\ti{A}_i, B/\ti{B}_i, A/ [ \ti{A}_i,\ti{A}_i]$ and $B/[\ti{B}_i,\ti{B}_i]$ are in general not groups, but we can view them as pointed sets.
We now consider the following commutative diagram of  exact sequences of pointed sets:
\[ \xymatrix{
0\ar[r]&
H_1(A;\Z[{B/\ti{B}_i}])\ar[r]\ar[d]^{\varphi}&
{A}/[\ti{A}_i,\ti{A}_i]\ar[d]^\varphi\ar[r]& {A/\ti{A}_i} \ar[d]^{\varphi}\ar[r]&1\\
 0\ar[r]& H_1(B;\Z[{B/\ti{B}_i}])
 \ar[r]& {B}/[\ti{B}_i,\ti{B}_i]\ar[r]&
{B/\ti{B}_i} \ar[r]&1.} \]
Recall that the map on the right is a bijection. It now follows from the 5--lemma for exact sequences of pointed sets that the middle map is a bijection if and only if the left hand map is a bijection.
\end{proof}

We will several times make use of the following corollary which can be deduced from Lemma \ref{lem:metabany} the same way
as Corollary \ref{cor:group1} is deduced from Lemma \ref{lem:group1}.

\begin{corollary}\label{cor:metabany}\label{cor:metab}
Let $\varphi :A\to B$ be a monomorphism of finitely generated groups,
and assume we are given a homomorphism $\beta:B \to G$ to a finite group $G$.
Then
 \[ \varphi_{*}:H_i(A;\Z[G])\to H_i(B;\Z[G]), \ \ i = 0,1 \] is an isomorphism
 if and only if the following two conditions hold:
\bn
\item $ \im\{A\to B\to G\}=\im\{B\to G\}$,
\item $\varphi$ induces an
 isomorphism
\[ A/ [ \ker(\beta\circ \varphi),\ker(\beta\circ \varphi)]\to B/[\ker(\beta),\ker(\beta)].\]
\en
\end{corollary}

Under extra conditions we can also give a criterion for a map between groups to induce an isomorphism of second homology groups.

\begin{lemma}\label{lem:twih2}
Let $\varphi:A\to B$ be a homomorphism between two groups such that $X=K(A,1)$ and $Y=K(B,1)$
are finite 2--complexes with vanishing Euler characteristic.
Let $\b:B\to G$ be a homomorphism to a finite group
 such that
 \[ \varphi_{*}:H_i(A;\Z[G])\to H_i(B;\Z[G]), \ \ i = 0,1 \] is an isomorphism,
 then
 \[ \varphi_{*}:H_2(A;\Z[G])\to H_2(B;\Z[G]) \] is also an isomorphism.
\end{lemma}

\begin{proof}
We can and will view $X$ as a subcomplex of $Y$.
It suffices to show that $H_2(Y,X;\Z[G])=0$.  Note that our assumption implies that $H_i(Y,X;\Z[G])=0$ for $i=0,1$.
 Now note that $H_2(Y,X;\Z[G])$ is a submodule of $C_2(Y,X;\Z[G])$, in particular $H_2(Y,X;\Z[G])$ is a free $\Z$--module. We therefore only have to show that $\rank H_2(Y,X;\Z[G])=0$.
Now note that
\[\ba{rcl} \rank H_2(Y,X;\Z[G])&=&\rank H_2(Y,X;\Z[G])-\rank H_1(Y,X;\Z[G])+\rank H_0(Y,X;\Z[G])\\
&=& |G| \chi(Y,X)\\
&=&|G|( \chi(Y)-\chi(X))\\
&=&0.\ea \]
\end{proof}

We conclude this section with the following lemma.

\begin{lemma}\label{lem:alsoiso}
Let $\varphi :A\to B$ be a homomorphism.
Let $\hat{B}\subset \ti{B}\subset B$ be two  subgroups.
Suppose that $\hat{B}\subset B$ is normal.
We write $\hat{A}:=\varphi^{-1}(\hat{B})$ and $\ti{A}:=\varphi^{-1}(\ti{B})$.
Assume that
\[ \varphi:A/ \hat{A} \to B/\hat{B} \mbox{ and } \varphi:A/ [ \hat{A},\hat{A}  ]\to B/[\hat{B},\hat{B}]\]
are bijections, then
\[  \varphi:A/ \ti{A} \to B/\ti{B} \mbox{ and } \varphi:A/ [ \ti{A},\ti{A}  ]\to B/[\ti{B},\ti{B}]\]
are also bijections.
\end{lemma}

\begin{proof}
In the following let $n=0$ or $n=1$. Suppose that  $\varphi:A/ \hat{A}^{(n)} \to B/\hat{B}^{(n)}$ is a bijection.
Note that $\hat{A}^{(n)}\subset A$ and $\hat{B}^{(n)}\subset B$ are normal, in particular  $\varphi:A/ \hat{A}^{(n)} \to B/\hat{B}^{(n)}$ is in fact an isomorphism.
We have to show that $ \varphi:A/ \ti{A}^{(n)}  \to B/\ti{B}^{(n)}$
is  a bijection.

\begin{claim}
The map $\varphi$ induces a bijection $\ti{A}^{(n)}/\hat{A}^{(n)}\to \ti{B}^{(n)}/\hat{B}^{(n)}$.
\end{claim}

We write $\ol{A}:=A/\hat{A}^{(n)}$, $\ol{B}:=B/\hat{B}^{(n)}$ and we denote by $\ol{\varphi}:\ol{A}\to \ol{B}$ the induced map which by assumption is an isomorphism.
We denote by $\ol{H}$ the subgroup $\ti{B}/\hat{B}^{(n)}\subset \ol{B}$.
Note that $\ol{\varphi}$ restricts to isomorphisms
$\ol{\varphi}^{-1}(\ol{H})\to \ol{H}$ and $\ol{\varphi}^{-1}(\ol{H}^{(n)})\to \ol{H}^{(n)}$.
Since $\ol{\varphi}^{-1}$ is an isomorphism it follows  that $\left(\ol{\varphi}^{-1}(\ol{H})\right)^{(n)}=\ol{\varphi}^{-1}\big(\ol{H}^{(n)}\big)$.
Now recall that $\ol{H}=\ti{B}/\hat{B}^{(n)}$, hence
$\ol{H}^{(n)}=\ti{B}^{(n)}/\hat{B}^{(n)}$.
We clearly have $\ol{\varphi}^{-1}(\ol{H})=\ti{A}/\hat{A}^{(n)}$ and therefore
$\ol{\varphi}^{-1}(\ol{H})^{(n)}=\ti{A}^{(n)}/\hat{A}^{(n)}$. This shows that the isomorphism $\ol{\varphi}:\ol{\varphi}^{-1}\big(\ol{H}^{(n)}\big)\to \ol{H}^{(n)}$
is precisely the desired isomorphism $\ti{A}^{(n)}/\hat{A}^{(n)}\to \ti{B}^{(n)}/\hat{B}^{(n)}$. This concludes the proof of the claim.

Now consider the following commutative diagram of short exact sequences of pointed sets:
\[
\xymatrix{ 1 \ar[r] & \ti{A}^{(n)}/\hat{A}^{(n)}\ar[d]^{\varphi}\ar[r] &A/\hat{A}^{(n)}\ar[d]^{\varphi}\ar[r]&A/\ti{A}^{(n)}\ar[d]^{\varphi}\ar[r]&1 \\
1 \ar[r] & \ti{B}^{(n)}/\hat{B}^{(n)}\ar[r] &B/\hat{B}^{(n)}\ar[r]&B/\ti{B}^{(n)}\ar[r]&1.}\]
The middle vertical map is a bijection by assumption and we just verified that the vertical map on the left is a bijection.
It now follows from the 5--Lemma for exact sequences of pointed sets that the vertical map on the right is also a bijection.
\end{proof}

\subsection{Twisted Alexander polynomials} \label{sectionufd}

In this section we are going to recall the definition of twisted Alexander polynomials. These were introduced,
for the case of knots, by Xiao-Song Lin in 1990 (published in \cite{L01}), and his definition
was later generalized to $3$--manifolds by Wada \cite{Wa94}, Kirk--Livingston \cite{KL99} and
Cha \cite{Ch03}.

Let $N$ be a compact manifold.
Let $R$ be a commutative, Noetherian unique  factorization domain
(in our applications $R=\Z$ or $R=\F_p$, the finite field with $p$ elements)
and $V$ a finite free $R$--module
Let $\a:\pi_1(N)\to \aut_{R}(V)$
a representation and let $\phi\in H^1(N;\Z)=\hom(\pi_1(N),\Z)$ a nontrivial element.
We write $V\otimes_R \rt=:\vt$.
Then $\a$ and $\phi$ give rise to
a representation $\a\otimes \phi:\pi_1(N)\to \aut_{\rt}(\vt)$
as follows:
\[ \big((\a\otimes \phi)(g)\big) (v\otimes p):= (\a(g)\cdot v)\otimes (\phi(g)\cdot p) = (\a(g) \cdot v)\otimes (t^{\phi(g)}p), \]
where $g\in \pi_1(N), v\otimes p \in V\otimes_R \rt = \vt$.

Note that $N$ is homotopy equivalent to a finite CW--complex, which, by abuse of notation, we also denote by $N$.
Then we consider
$C_*(\ti{N})\otimes_{\Z[\pi_1(N)]}\vt$
which is a complex of  finitely generated $\rt$--modules.
Since $\rt$ is Noetherian it follows that for any $i$ the $\rt$--module $H_i(N;\vt)$ is
a finitely
presented $\rt$--module. This means
$H_i(N;\vt)$ has a free $\rt$--resolution
\[ \rt^{r_i} \xrightarrow{S_i} \rt^{s_i} \to H_i(N;\vt) \to 0. \]
 Without loss of generality we can assume that $r_i\geq s_i$.

\begin{definition} \label{def:alex} The \emph{$i$--th twisted Alexander polynomial} of $(N,\a,\phi)$ is defined
to be  the order of the $\rt$--module $H_i(N;\vt)$, i.e. the greatest common divisor
(which exists since $\rt$ is a UFD as well) of the
$s_i\times s_i$ minors of the $s_i\times r_i$--matrix $S_i$. It is denoted by $\Delta_{N,\phi,i}^{\a}\in
\rt$.
\end{definition}

Note that $\Delta_{N,\phi,i}^{\a}\in
\rt$ is well--defined up to a unit in $\rt$, i.e. up to an element of the form $rt^i$ where $r$ is a unit
in $R$ and $i\in \Z$. We say that $f\in\rt$ is \textit{monic} if its top coefficient is a unit in $R$.
Given a nontrivial $f=\sum_{i=r}^s a_st^i$ with $a_r\ne 0, a_s\ne 0$ we write
$\deg f=s-r$. For $f=0$ we write $\deg(f)=-\infty$. Note that $\deg \Delta_{N,\phi,i}^{\a}$ is well--defined.

We now write $\pi=\pi_1(N)$. If we are given a homomorphism $\a:\pi\to G$ to a finite group, then this gives rise to a
finite dimensional representation of $\pi$, that we will denote by $\a:\pi\to \aut_{R}(R[G])$ as well.
In the case that we have a finite index subgroup $\ti{\pi}\subset \pi$
we get a finite dimensional representation $\pi\to \aut_{R}(R[\pi/\ti{\pi}])$ given by left--multiplication.
When $R = \Z$, the resulting twisted Alexander polynomials will be denoted by $\Delta_{N,\phi,i}^{\pi/\ti{\pi}}\in \zt$, while for $R = \F_p$ we will use the notation $\Delta_{N,\phi,i}^{\pi/\ti{\pi},p}\in \F_p\tpm$. See \cite{FV08a} for the relation between these polynomials.

Finally, in the case that $\a:\pi\to \gl(1,\Z)$ is the trivial representation we drop the $\a$ from the notation,
and in the case that $i=1$ we drop the subscript ``$,1$'' from the notation.

We summarize some of the main properties of twisted Alexander polynomials in the following lemma.
It is a consequence of \cite[Lemma~3.3~and~3.4]{FV08a}
and   \cite[Proposition~2.5]{FK06}.

\begin{lemma} \label{lem:twiprop}
Let $N$ be a 3--manifold with empty or toroidal boundary. Let $\phi \in H^1(N;\Z)$ nontrivial
and $\ti{\pi}\subset \pi:=\pi_1(N)$ a finite index subgroup.
Denote by $\phi_{\ti{\pi}}$ the restriction of $\phi$ to $\ti{\pi}$,
 then the following hold:
 \bn
 \item $ \Delta_{N,\phi,0}^{\pi/\ti{\pi}} =(1-t^{div\, {\phi}_{\ti{\pi}}})$,
 \item if $\Delta_{N,\phi,1}^{\pi/\ti{\pi}}\ne 0$,
then $\Delta_{N,\phi,2}^{\pi/\ti{\pi}} =(1-t^{div \, \phi_{\ti{\pi}}})^{b_3(N)}$,
\item $ \Delta_{N,\phi,i}^{\pi/\ti{\pi}}=1$ for any $i\geq 3$.
\en
Assume we also have a subgroup $\pi'$ with  $\ti{\pi}\subset \pi' \subset \pi$.
Denote the covering of $N$ corresponding to $\pi'$ by $N'$ and
denote by $\phi'$ the restriction of $\phi$ to $\pi'$,
then
\[ \Delta_{N,\phi,i}^{\pi/\ti{\pi}}= \Delta_{N',\phi',i}^{\pi'/\ti{\pi}}\]
for any $i$. Finally note that the statements of the lemma also hold for the polynomial $\Delta_{N,\phi,i}^{\pi/\ti{\pi},p}\in \F_p\tpm$.
\end{lemma}

We also recall  the following well--known result (cf. e.g. \cite{Tu01}).

\begin{lemma} \label{lem:twiprop2}
Let $N$ be a 3--manifold with empty or toroidal boundary. Let $\phi \in H^1(N;\Z)$ nontrivial
and $\ti{\pi}\subset \pi:=\pi_1(N)$ a finite index subgroup.
Then given $i$ the following are equivalent:
\bn
\item $\Delta_{N,\phi,i}^{\pi/\ti{\pi}} \ne 0$,
\item $H_i(N;\Z[\pi/\ti{\pi}]\tpm)$ is $\zt$--torsion,
\item $H_i(N;\Q[\pi/\ti{\pi}]\tpm)$ is $\Q\tpm$--torsion,
\item the rank of the abelian group $H_i(N;\Z[\pi/\ti{\pi}]\tpm)$ is finite.
\en
In fact if any of the four conditions holds, then
\[  \deg\Delta_{N,\phi,i}^{\pi/\ti{\pi}} =\mbox{rank}\, H_i(N;\Z[\pi/\ti{\pi}]\tpm)=\dim \, H_i(N;\Q[\pi/\ti{\pi}]\tpm).\]
\end{lemma}
\subsection{Completions of groups}\label{section:pro}

Throughout the paper it is convenient to use the language of completions of groups.
 Although the proof of Theorem \ref{mainthm} does not explicitly require this terminology,
  group completions are the natural framework for these results.
We recall here the definitions and some basic facts, we refer to \cite[Window~4]{LS03} and \cite{Wi98,RZ00} for  proofs and for more information.

Let $\mathcal{C}$ be a variety of groups (cf. \cite[p.~20]{RZ00} for the definition).
Examples of varieties of pertinence to this paper are given by any one of the following:
\bn
\item finite groups;
\item $p$--groups for a prime $p$;
\item the variety $\fs(n)$ of finite solvable groups of derived length at most $n$;
\item the variety $\fs$ of finite solvable groups.
\en
In the following we equip a finitely generated group $A$ with its \emph{pro--$\CC$ topology}, this topology is
the translation invariant topology uniquely defined by taking as a fundamental system
of neighborhoods of the identity the collection of all normal subgroups of $A$ such that the quotient lies in $\CC$.
Note that in particular  all groups in $\mathcal{C}$ are endowed with the discrete topology.

Given a group $A$ denote by $\hat{A}_\mathcal{C}$ its pro--$\mathcal{C}$ completion, i.e. the inverse limit
\[ \hat{A}_\mathcal{C}=\underleftarrow{\lim}\, A \, /A_i \]
where $A_i$ runs through the inverse system determined by the collection
 of all normal subgroups of $A$ such that $A/A_i\in \mathcal{C}$.
 Then $\hat{A}_\mathcal{C}$, which we can view as a subgroup of the direct
product of all $A/A_i$, inherits a natural topology. Henceforth  by homomorphisms between groups we will mean a homomorphism which is continuous with respect to the above topologies.
Using the standard convention we refer to the pro--$\fs$ completion of a group as the prosolvable completion.

Note that by the assumption that $\mathcal{C}$ is a variety, the pro--$\mathcal{C}$ completion is
a covariant functor, i.e. given $\varphi:A\to B$ we get an induced homomorphism $\hat{\varphi}:\hat{A}_\mathcal{C}\to \hat{B}_\mathcal{C}$.

A group $A$ is called \textit{residually $\mathcal{C}$} if for any nontrivial $g\in A$ there exists a
homomorphism $\a:A\to G$ where $G\in \mathcal{C}$ such that $\a(g)\ne e$. It is easily seen that $A$ is
residually $\mathcal{C}$ if and only if the  map $A\to \hat{A}_\mathcal{C}$ is injective.
In particular, if we are given a homomorphism $\varphi:A\to B$ between residually ${\CC}$ groups $A,B$
such that $\hat{\varphi}:\hat{A}_{\CC}\to \hat{B}_{\CC}$
is an injection, then it follows from the following commutative diagram
\[ \xymatrix{ A\ar[d] \ar[r]^\varphi & B \ar[d] \\
\hat{A}_{\CC}\ar[r]^{\hat{\varphi}}&\hat{B}_{\CC} }\]
that $\varphi$ is injective as well.

The following well--known lemma gives sufficient and necessary
conditions for a homomorphism $\varphi:A\to B$ to induce  an isomorphism of pro--$\CC$ completions.

\begin{lemma} \label{lem:pro} Let $\mathcal{C}$ be a variety of groups and assume that there is a homomorphism $\varphi : A \to B$.
 Then the following are equivalent:
\bn
\item $\hat{\varphi}:\hat{A}_\mathcal{C}\to \hat{B}_\mathcal{C}$ is an isomorphism,
\item for any  $G\in \mathcal{C}$ the induced map
\[ \varphi^*:\hom(B,G)\to \hom(A,G) \]
is a bijection.
\en
\end{lemma}

We also note the following well--known lemma.

\begin{lemma}\label{lem:restrictiso}
Let $\CC$ be an extension--closed variety  and
let $\varphi:A\to B$ a homomorphism of finitely generated groups
which induces an isomorphism of pro--$\mathcal{C}$ completions. Then for any homomorphism $\beta:B\to G$ to a $\mathcal{C}$--group the restriction of $\varphi$ to $\ker(\beta \circ \varphi)\to \ker(\beta)$  induces an isomorphism of pro--$\mathcal{C}$ completions.
\end{lemma}

When a homomorphism $\varphi : A \to B$ of finitely generated groups induces an isomorphism of their pro--$\mathcal{C}$ completions, then we have a relation of the twisted homology with coefficients determined by ${\mathcal C}$--groups. More precisely, we have the following.

\begin{lemma} \label{lem:im}
\label{lem:twih1}
Let $\mathcal{C}$ be a variety of groups and let $\varphi : A \to B$ be a homomorphism of finitely generated groups
which induces an isomorphism of pro--$\mathcal{C}$ completions. Then for any homomorphism $\beta:B\to G$ to a $\mathcal{C}$--group the map $\varphi_{*} : H_0(A;\Z[G]) \to H_0(B;\Z[G])$ is an isomorphism. Furthermore, if
$\CC$ is an extension--closed variety containing all finite abelian groups, the map $\varphi_{*} : H_1(A;\Z[G]) \to H_1(B;\Z[G])$ is an isomorphism. \end{lemma}

\begin{proof}
Observe that, by Corollary \ref{cor:group1}, the first part of the statement is equivalent to the assertion that, for any element $\beta \in \hom(B,G)$, \[ \im\{\beta \circ \varphi : A \to G\} = \im\{\beta : B \to G\}. \]
Without loss of generality, we can reduce the proof of this isomorphism to the case
where $\beta$ is surjective. Denote $\alpha = \beta \circ \varphi \in \hom(A,G)$. Assume to the contrary
 that $\a(A) \subsetneq G$; then $\alpha \in \hom(A,\alpha(A)) \subset \hom(A,G)$ and as $\a(A) \in \mathcal{C}$ there exists by hypothesis a map
$\beta' \in \hom(B,\a(A)) \subset \hom(B,G)$ such that
$\alpha =  \beta' \circ \varphi$. Now the two maps $\beta, \beta' \in \hom(B,G)$ (that must differ as they have different image)
induce the same map $\a \in \hom(A,G)$, contradicting the bijectivity of $\hom(B,G)$ and $\hom(A,G)$.

We now turn to the proof of  the second part of the statement. Let $\b:B\to G$ a homomorphism to a $\CC$--group. Again, without loss of generality, we can assume that $\b:B\to G$ is surjective.
Note that by the above the homomorphism $\beta\circ\varphi:A\to G$ is surjective as well.
We now write $B' = \mbox{Ker}( \beta)$ and $A' = \mbox{Ker}( \beta \circ \varphi)$.
 By Shapiro's Lemma, we have the commutative diagram
 \[ \xymatrix{ H_1(A';\Z) \ar[r]^-\cong \ar[d] &H_1(A;\Z[G])\ar[d]\\
  H_1(B';\Z)\ar[r]^-\cong& H_1(B;\Z[G]).}\]
The claim amounts therefore to showing that  the map
$\varphi_{*} : H_1(A';\Z) \to H_1(B';\Z)$ is an isomorphism.
As $A$ and $B$ are finitely generated, $A'$ and $B'$ are finitely generated as well. When $\CC$ is extension closed, and contains all finite abelian groups, Lemma \ref{lem:restrictiso} asserts that the map $\varphi$ induces
a bijection between $\hom(B',\Gamma)$ and $\hom(A',\Gamma)$ for any finite abelian group $\Gamma$; the desired isomorphism easily follows.\end{proof}

\section{Monic twisted Alexander polynomials and solvable groups}\label{section:monic}

The aim of this section is to prove Proposition \ref{thm:finitesolvable}.

\subsection{Preliminary results}

We will often make use of the following proposition (cf. \cite[Section~4~and~Proposition~6.1]{McM02}).

\begin{proposition} \label{prop:mcm}
Let $N$ be a $3$--manifold  with empty or toroidal boundary and let
$\phi \in H^{1}(N;\Z)$ a primitive class.
If  $\Delta_{N,\phi}\ne 0$, then there exists a connected Thurston norm minimizing surface
$\S$ dual to $\phi$.
\end{proposition}

Given a connected oriented surface $\S\subset N$ we will adopt the following conventions for the rest of the paper.
We choose a  neighborhood $\S\times [-1,1]\subset N$ and write $\nu \S=\S\times (-1,1)$.
Let $M := N \setminus \nu \Sigma$; we will write $\S^\pm=\S\times \{\pm 1\}\subset \partial M$, and we will denote the inclusion induced maps $\S \to \S^\pm\subset M$ by $\i_\pm$.

We pick a base point in $M$ and endow $N$ with the same base point. Also, we
pick a base point for $\S$ and endow $\S^\pm$ with the corresponding base points.  With these choices made, we will write $A = \pi_1(\S)$ and $B = \pi_1(M)$.
We also  pick paths in $M$ connecting the base point of $M$ with the base points
of $\S^-$ and $\S^+$.
We now have inclusion induced maps $\i_\pm : A \to B$ for either inclusion of $\S$ in $M$ and, using the constant path, a map $\pi_1(M) \to \pi_1(N)$.
Under the assumption that $\S$ is incompressible (in particular, whenever $\S$ is Thurston norm minimizing)
these maps are injective. Since $M$ and $N$ have the same base point we can view $B$ canonically as a subgroup of $\pi_1(N)$.

Before we state the first proposition we have to introduce a few more definitions.
Let $N$ be a $3$--manifold with empty or toroidal boundary and let $\phi \in H^{1}(N;\Z)$ a nontrivial class.
Let $\tipi\subset \pi$ be a finite index subgroup.
We denote by $\phi_{\tipi}$  the restriction of $\phi\in H^1(N;\Z)=\hom(\pi,\Z)$ to $\tipi$.
Note that $\phi_{\tipi}$ is necessarily non--trivial.
We say that  $\tipi\subset \pi$ has Property (M) if
the twisted Alexander polynomial $\Delta_{N,\phi}^{\pi/\tipi}\in \zt$ is monic
and if \[ \deg(\Delta_{N,\phi}^{\pi/\tipi})= [\pi:\tipi] \, \|\phi\|_{T} + (1+b_3(N)) \div  \phi_{\tipi}\] holds. Note that a pair $(N,\phi)$ satisfies Condition ($*$) if and only if Property (M) is satisfied by all normal subgroups of $\pi_1(N)$.

The following proposition  is the key tool for translating information on twisted Alexander polynomials
into information on the maps $\i_\pm:A\to B$.
The proposition is well known in the classical case.
In the case of normal subgroups a proof for the `only if' direction of the proposition is given by combining  \cite[Section~8]{FV08a} with \cite[Section~4]{FV08b}.

\begin{proposition} \label{prop:sameh0h1}
Let $N$ be a 3--manifold with empty or toroidal boundary with $N\ne S^1\times D^2, N\ne S^1\times S^2$. Let $\phi  \in
H^{1}(N;\Z)$ a primitive class which is dual to a
connected Thurston norm
minimizing surface $\S$.
Let $\tipi\subset \pi$ be a finite index subgroup. Then $\tipi$ has Property (M) if and only if
the maps
$\i_{\pm}:H_i(A;\Z[{\pi/\tipi}])\to H_i(B;\Z[{\pi/\tipi}])$
are isomorphisms for $i=0,1$.
\end{proposition}

\begin{proof}
Let $R=\Z$ or $R=\F_p$ with $p$ a prime. We have canonical isomorphisms  $H_i(\S;R[\pi/\tipi])\cong H_i(A;R[\pi/\tipi])$ and
$H_i(M;R[\pi/\tipi])\cong H_i(B;R[\pi/\tipi])$  for $i=0,1$.
It follows from \cite[Proposition~3.2]{FK06}  that splitting $N$ along $\S$ gives rise to
 the following Mayer--Vietoris type exact sequence \[ \ba{cccccccccccccc} \hspace{-0.1cm}&\hspace{-0.1cm}&\hspace{-0.1cm}&\hspace{-0.1cm}\dots \hspace{-0.1cm}&\hspace{-0.1cm}\to\hspace{-0.1cm}&\hspace{-0.1cm} H_2(N;R[\pi/\tipi]\tpm)\\[0.1cm]
 \to \hspace{-0.1cm}&\hspace{-0.1cm}H_1(A;R[\pi/\tipi])\otimes \rt \hspace{-0.1cm}&\hspace{-0.1cm}\xrightarrow{t\i_+-\i_-}
\hspace{-0.1cm}&\hspace{-0.1cm}H_1(B;R[\pi/\tipi])\otimes \rt \hspace{-0.1cm}&\hspace{-0.1cm}\to\hspace{-0.1cm}&\hspace{-0.1cm}
H_1(N;R[\pi/\tipi]\tpm)\hspace{-0.1cm}&\hspace{-0.1cm}\to\hspace{-0.1cm}&\hspace{-0.1cm}\\[0.1cm]
\to \hspace{-0.1cm}&\hspace{-0.1cm}H_0(A;R[\pi/\tipi])\otimes \rt \hspace{-0.1cm}&\hspace{-0.1cm}\xrightarrow{t\i_+-\i_-}
\hspace{-0.1cm}&\hspace{-0.1cm}H_0(B;R[\pi/\tipi])\otimes \rt \hspace{-0.1cm}&\hspace{-0.1cm}\to\hspace{-0.1cm}&\hspace{-0.1cm}
H_0(N;R[\pi/\tipi]\tpm)\hspace{-0.1cm}&\hspace{-0.1cm}\to\hspace{-0.1cm}&\hspace{-0.1cm} 0.
\ea \]
which we refer to as the Mayer--Vietoris sequence of $(N,\S)$ with $R[\pi/\tipi]\tpm$--coefficients.
First note that by Shapiro's lemma the groups $H_i(A;R[\pi/\tipi])$ are the $i$--th  homology with $R$--coefficients of a (possibly) disconnected surface.
It follows that  $H_i(A;R[\pi/\tipi])$ is a free $R$--module, in particular
 the $\rt$--modules $H_i(A;R[\pi/\tipi])\otimes \rt$ are free $\rt$--modules. We will several times make use of the observation that if $H_i(N;R[\pi/\tipi]\tpm)$ is $\rt$--torsion, then
the map $H_i(N;R[\pi/\tipi]\tpm)\to H_{i-1}(A;R[\pi/\tipi])\otimes \rt$ is necessarily zero.

We first assume that $\tipi$ has Property (M).
 Since  $\Delta_{N,\phi}^{\pi/\tipi}\ne 0$ we have  that the module $H_1(N;\Z[\pi/\tipi]\tpm)\otimes_{\zt} \Q(t)$ is trivial.
Note that by Lemma  \ref{lem:twiprop2} we have that $H_0(N;\Z[\pi/\tipi]\tpm)$  is also $\zt$--torsion.
  We now consider the Mayer--Vietoris sequence of $(N,\S)$ with $\Z[\pi/\tipi]\tpm$--coefficients.
 Tensoring the
 exact sequence with $\Q(t)$ we  see  that
\[ \ba{rcl}\rank_{\Z}(H_0(A;\Z[\pi/\tipi]))&=&\rank_{\Q(t)}(H_0(A;\Z[\pi/\tipi])\otimes_{\Z}\Q(t))\\
&=& \rank_{\Q(t)}(H_0(B;\Z[\pi/\tipi])\otimes_\Z \Q(t))=
\rank_{\Z}(H_0(B;\Z[\pi/\tipi])).\ea \]
Using this observation and using Lemma \ref{lem:group1} we see that the maps $\i_\pm:H_0(A;\zpp)\to H_0(B;\zpp)$ are epimorphism
between free abelian groups of the same rank. Hence the maps are in fact isomorphisms.

In order to prove that the maps  $\i_\pm:H_1(A;\zpp)\to H_1(B;\zpp)$ are isomorphisms
we first consider  the following claim.

\begin{claim}
$H_1(A;\Z[\pi/\tipi])$ and $H_1(B;\Z[\pi/\tipi])$ are free abelian groups of the same rank.
\end{claim}

Let $p$ be a prime.
We consider the Mayer--Vietoris sequence of $(N,\S)$ with $\F_p[\pi/\tipi]\tpm$--coefficients.
Denote by $\Delta_{N,\phi}^{{\pi/\tipi,p}} \in \fpt$ the twisted Alexander polynomial with coefficients in $\F_p$. It follows from $\Delta^{\pi/\tipi}_{N,\phi}$ monic  and from \cite[Proposition~6.1]{FV08a} that
$\Delta_{N,\phi}^{{\pi/\tipi,p}}\ne 0\in \fpt$. Furthermore by Lemma
\ref{lem:twiprop} we have that $\Delta_{N,\phi,2}^{{\pi/\tipi,p}}\ne 0\in \fpt$. In particular  $H_i(N;\fp[\pi/\tipi][t^{\pm 1}])$ is $\fpt$--torsion for $i=1,2$.
It follows from the fact that $H_i(A;\fp[\pi/\tipi])\otimes_{\F_p} \fpt$ is a free $\fpt$--module and the above observation
that $H_i(N;\fp[\pi/\tipi][t^{\pm 1}])$ is $\fpt$--torsion for $i=1,2$ that the Mayer--Vietoris sequence gives rise to the following
 short exact sequence
\[ 0\to H_1(A;\fp[\pi/\tipi])\otimes \fpt \xrightarrow{t \i_+-\i_-}  H_1(B;\fp[\pi/\tipi])\otimes
\fpt \to  H_1(N;\fp[\pi/\tipi][t^{\pm 1}])\to 0.\] Tensoring with $\F_p(t)$ we see that in
particular $H_1(A;\fp[\pi/\tipi])\cong H_1(B;\fp[\pi/\tipi])$  as
$\F_p$--vector spaces. The homology group $H_0(A;\Z[\pi/\tipi])$ is $\Z$--torsion free. It
follows from the universal coefficient theorem applied to the complex of $\Z$--modules
$C_*(\ti{\S})\otimes_{\Z[A]}{\Z[\pi/\tipi]}$ that
\[ H_1(A;\Z[\pi/\tipi])\otimes_\Z \F_p \cong H_1(A;\F_p[\pi/\tipi])\]
for every prime $p$. The same statement holds for $B$. Combining our results we see that
$H_1(A;\Z[\pi/\tipi])\otimes_\Z \F_p$ and $H_1(B;\Z[\pi/\tipi])\otimes_\Z \F_p$ are isomorphic for any prime $p$.
Since  $H_1(A;\Z[\pi/\tipi])$ is free abelian it
follows that $H_1(A;\Z[\pi/\tipi])\cong H_1(B;\Z[\pi/\tipi])$.
This completes the proof of the claim.

In the following we equip the free $\Z$--modules $H_1(A;\Z[\pi/\tipi])$ and $H_1(B;\Z[\pi/\tipi])$ with a choice of basis.
We now study the Mayer--Vietoris sequence for $(N,\S)$ with $\Z[\pi/\tipi]\tpm$--coefficients. Using an argument similar to the above we see that it gives rise to  the following exact sequence
\[ H_1(A;\Z[\pi/\tipi])\otimes \zt \xrightarrow{t \i_+-\i_-} H_1(B;\Z[\pi/\tipi])\otimes \zt \to H_1(N;\Z[\pi/\tipi]\tpm)\to 0.\]
Since $H_1(A;\Z[\pi/\tipi])$ and $H_1(B;\Z[\pi/\tipi])$ are free abelian groups of the same rank it
follows that the above exact sequence is a resolution of $H_1(N;\Z[\pi/\tipi]\tpm)$ by free $\zt$--modules and  that $\Delta^{\pi/\tipi}_{N,\phi}=\det(t \i_+-\i_-)$.
Recall that Property (M) states in particular  that
\be \label{equ:deg} \deg\Delta_{N,\phi}^{{\pi/\tipi}} = |\pi/\tipi| \, \|\phi\|_{T} + (1+b_3(N)) \div \phi_{\tipi}.\ee
Recall that we assumed that $N\ne S^1\times D^2$ and $N\ne S^1\times S^2$, in particular $\chi(\S)\leq 0$ and therefore $-\chi(\S)=||\phi||_T$.
Writing $b_i=\rank_{\Z}(H_i(\S;\Z[\pi/\tipi]))=\rank_{\Z}(H_i(A;\Z[\pi/\tipi]))$ a standard Euler characteristic
argument now shows that
\[-b_0+b_1-b_2=-|\pi/\tipi|\chi(\S)=|\pi/\tipi|\cdot ||\phi||_T.\]
By
\cite[Lemma~2.2]{FK06} we have $b_i=\deg \Delta^{{\pi/\tipi}}_{N,\phi,i}$ for $i=0$ and $i=2$.
We also have $\deg \Delta^{{\pi/\tipi}}_{N,{\phi},0}=\div \phi_{\pi/\tipi}$ and
 $\deg \Delta^{{\pi/\tipi}}_{N,{\phi},2}=b_3(N) \div \phi_{\tipi}$ by Lemma \ref{lem:twiprop}.
  Combining these facts with (\ref{equ:deg}) we conclude that
  $\deg \Delta^{{\pi/\tipi}}_{N,\phi}=b_1$.
So we now have $\deg(\det(t \i_+-\i_-))=b_1$. Since
$\i_+$ and $\i_-$ are $b_1\times b_1$ matrices over $\Z$
it now follows  that $\det(\i_+)$ equals the top coefficient
of $\Delta^{\pi/\tipi}_{N,\phi}$, which by Property (M) equals $\pm 1$. By the symmetry of twisted Alexander polynomials we have that
the bottom coefficient of $\Delta^{\pi/\tipi}_{N,\phi}$ also equals $\pm 1$, we deduce that $\det(\i_-)=\pm 1$. This
shows that $\i_+,\i_-: H_1(A;\Z[\pi/\tipi])\to H_1(B;\Z[\pi/\tipi])$ are
isomorphisms. We thus showed that if $\tipi$ has Property (M), then
the maps
$\i:H_i(A;\Z[{\pi/\tipi}])\to H_i(B;\Z[{\pi/\tipi}])$
are isomorphisms for $i=0,1$.

Now assume that we are given a finite index subgroup $\tipi\subset \pi$ such that
the maps
$\i_\pm:H_i(A;\Z[{\pi/\tipi}])\to H_i(B;\Z[{\pi/\tipi}])$
are isomorphisms for $i=0,1$. It follows from the assumption that $\i_\pm:H_0(A;\Z[{\pi/\tipi}])\to H_0(B;\Z[{\pi/\tipi}])$
are isomorphisms that the map
\[ H_0(A;\Z[\pi/\tipi])\otimes \zt \xrightarrow{t\i_+-\i_-}
H_0(B;\Z[\pi/\tipi])\]
is injective.  In particular the  Mayer--Vietoris sequence of $(N,\S)$
with $\Z[\pi/\tipi]\tpm$--coefficients gives rise to  the following exact sequence
\[ H_1(A;\Z[\pi/\tipi])\otimes \zt \xrightarrow{t \i_+-\i_-} H_1(B;\Z[\pi/\tipi])\otimes \zt \to H_1(N;\Z[\pi/\tipi]\tpm)\to 0.\]
As above $H_1(A;\Z[\pi/\tipi])$ is a free abelian group and  by our assumption
$H_1(B;\Z[\pi/\tipi])\cong H_1(A;\Z[\pi/\tipi])$ is also free abelian. In particular the above exact sequence defines a presentation
for $H_1(N;\Z[\pi/\tipi]\tpm)$ and we deduce that
\[ \Delta_{N,\phi}^{\pi/\tipi}=\det(t\i_+-\i_-).\]
Since $\i_-$ and $\i_+$ are isomorphisms it follows that $\Delta_{N,\phi}^{\pi/\tipi}$ is monic of degree
$b_1$. An argument similar to the above now shows that
\[  \deg\Delta_{N,\phi}^{{\pi/\tipi}} = |\pi/\tipi| \, \|\phi\|_{T} + (1+b_3(N)) \div \phi_{\tipi}.\]
\end{proof}

\subsection{Finite solvable quotients}
Given a solvable group $S$
we denote by $\ell(S)$ its derived length,
i.e. the length of the shortest decomposition into abelian groups. Put differently, $\ell(S)$ is the minimal number such that $S^{(\ell(S))}=\{e\}$.
Note that $\ell(S)=0$ if and only if $S=\{e\}$.

For sake of comprehension, we briefly recall the notation. We are considering a $3$--manifold $N$ with empty or
toroidal boundary, and we fix a primitive class $\phi \in H^1(N;\Z)$. We denote by $\S$ a connected Thurston
norm minimizing surface dual to $\phi$, and write $A = \pi_1(\S)$ and $B = \pi_1(M)$ (with $M = N \setminus \nu \S$)
and we denote the two inclusion induced maps $A \to B$ with $\i_{\pm}$.
We also write $\pi=\pi_1(N)$. Note that $\pi=\ll B,t | \i_-(A)=t\i_+(A)t^{-1}\rr$.

Given $n\in \N\cup \{0\}$ we denote by $\statefs(n)$ the statement that for any finite solvable group $S$
with $\ell(S)\leq n$ we have that for $\i=\i_-,\i_+$ the map
\[ \i^*:\hom(B,S)\to \hom(A,S)\]
is a bijection. This is equivalent by Lemma \ref{lem:pro} to assert that $\i : A \to B$ induces an isomorphism of pro--$\mathcal{FS}(n)$ completions. Recall that by
Corollary \ref{cor:group1} and Lemma \ref{lem:im} statement $\statefs(n)$ implies then that for any homomorphism $\b:B\to S$ to a finite solvable group $S$ with $\ell(S)\leq n$ we have
$\im\{\b\circ \i:A\to B\to S\}=\im\{\b:B\to S\}$.

Our goal is to show that $\statefs(n)$ holds for all $n$. We will show this by induction on $n$.
For the induction argument we use the following  auxiliary statement:
Given $n\in \N\cup \{0\}$ we denote by $\stateh(n)$ the statement that for any homomorphism $\b:B\to S$ where $S$ is finite solvable with $\ell(S)\leq n$ we have that for $\i=\i_{-},\i_+$
the homomorphism \[ \i_{*} :H_1(A;\Z[S])\to H_1(B;\Z[S]) \]
is an isomorphism.

In the next two sections we will  prove the following two propositions:

\begin{proposition} \label{prop:ind1}
If $\stateh(n)$ and $\statefs(n)$ hold, then $\statefs(n+1)$ holds as well.
\end{proposition}

\begin{proposition} \label{prop:ind2}
Assume that  $(N,\phi)$ satisfies Condition  $(*)$.
If $\statefs(n)$ holds, then $\stateh(n)$ holds as well.
\end{proposition}

We can now prove the following corollary, which amounts to  Proposition \ref{thm:finitesolvable}.

\begin{corollary}\label{cor:finitesolvable}
Assume that $(N,\phi)$ satisfies Condition  $(*)$ and that $\phi$ is primitive.
Let $\S\subset N$ be a connected Thurston norm minimizing surface dual to $\phi$ and let $\i: A \to B$ be one of the two injections. Then for any finite solvable  group $G$ the map
\[ \hom(B,G) \xrightarrow{\i^*} \hom(A,G)\]
is a bijection, i.e. $\i : A \to B$ induces an isomorphism of prosolvable completions.
\end{corollary}

\begin{proof} The condition $\statefs(0)$ holds by \textit{fiat}.
It follows  from Proposition \ref{prop:sameh0h1} applied to the trivial group that if $(N,\phi)$ satisfies Condition ($*$), then $\i_\pm:H_1(A;\Z)\to H_1(B;\Z)$ are isomorphisms, i.e.  $\stateh(0)$ holds.
The combination of Propositions \ref{prop:ind1} and \ref{prop:ind2}
then shows that $\stateh(n)$ and $\statefs(n)$ hold for all $n$.
The corollary is now immediate.
\end{proof}

\subsection{Proof of Proposition \ref{prop:ind1}}

In this section we will prove Proposition \ref{prop:ind1}.
Let $\i=\i_-$ or $\i=\i_+$.
Since $\statefs(n)$ holds we only have to consider the case of  $G$  a finite solvable group with $\ell(G)=n+1$. By definition $G$ fits into a short exact sequence
\[ 1\to I\to G\to S\to 1,\]
where $I=G^{(n)}$ is finite abelian and $S=G/G^{(n)}$ finite solvable with $\ell(S)=n$.

We will construct a map $\Phi:\hom(A,G)\to \hom(B,G)$ which is an inverse to $\i^*:\hom(B,G)\to \hom(A,G)$.
Let $\a:A\to G$ be a homomorphism. Without loss of generality we can assume that $\a$ is an epimorphism. Denote $A\xrightarrow{\a}G\to S$ by $\a'$ and denote
the map $A\to A/\ker(\a')^{(1)}$ by $\rho$.
Note that $\a$ sends $\ker(\a')$ to the abelian group $I$, hence $\a$  vanishes on $\ker(\a')^{(1)}$.
This shows that $\a$ factors through $\rho$, in particular $\a=\psi \circ \rho$ for some
$\psi : A/\ker(\a')^{(1)}  \to G$.

 Recall that $\ell(S)=n$, therefore by $\statefs(n)$ we have that $\a':A\to S$ equals $\i^*(\b')$ for some
$\b':B\to S$. By Lemma \ref{lem:im}, $\statefs(n)$ guarantees that $i_{*} : H_0(A;\Z[S]) \to H_0(B;\Z[S])$ is an isomorphism; on the other hand $\stateh(n)$ asserts that $i_{*} : H_1(A;\Z[S]) \to H_1(B;\Z[S])$ is an isomorphism as well. By Corollary \ref{cor:metab} this implies that $\i$ induces an isomorphism
\[ \i:A/\ker(\a')^{(1)}\xrightarrow{\cong} B/\ker(\b')^{(1)}.\]
The various homomorphisms can be summarized in the following commutative diagram:
\[
\xymatrix{
A \ar[ddd]^{\a} \ar[rrr]^{\i} \ar[dr]^{\a'} \ar[ddr]_\rho &&& B\ar[dl]_{\b'}\ar[ddl]\\
&S\ar[r]^= &S &\\
& A/\ker(\a')^{(1)}\ar[r]^{\cong}_{\i} \ar@{->>}[dl]^{\psi}\ar@{->>}[u]_{\a'}&B/\ker(\b')^{(1)}\ar@{->>}[u]^{\b'}&\\
G.&&&}\]

Now we define $\Phi(\a)\in \hom(B,G)$ to be the homomorphism
\[  B\to  B/\ker(\b')^{(1)} \xrightarrow{\i^{-1}} A/\ker(\a')^{(1)}\xrightarrow{\psi} G.\]
It is now straightforward  to check that $\Phi$ and $\i^*$ are inverses to each other.

\subsection{Proof of Proposition \ref{prop:ind2}} \label{section:defch}
In this section we will prove Proposition \ref{prop:ind2}.
So let $\b:B\to S$ be a homomorphism to a finite solvable group $S$
with $\ell(S)\leq n$. If $\b$ extends to $\pi_1(N)$, $\stateh(n)$ will follow immediately
from Proposition \ref{prop:sameh0h1}. In general $\b$ though will not extend; however using $\statefs(n)$ we will construct a homomorphism $\pi=\ll B,t | \i_-(A)=t\i_+(A)t^{-1}\rr \to G$ to a finite group $G$ `which contains $\b:B\to S$' to get the required isomorphism.

We first need some notation.
Given groups $C$ and $H$ we define
\[ C(H)=\bigcap\limits_{\g\in Hom(C,H)}\ker(\g).\]
We summarize a few properties of $C(H)\subset C$ in the following lemma.

\begin{lemma}\label{lem:chprop}
Let $C$ be a finitely generated group. Then the subgroup $C(H)\subset C$ has the following properties:
\bn
\item $C(H)\subset C$ is normal and characteristic.
\item If $H$ is finite and solvable, then $C/C(H)$ is finite and solvable with $\ell(C/C(H))\leq \ell(H)$.
\en
\end{lemma}

\begin{proof}
Statement (1) is immediate. To prove the rest, consider the  injection
\[ C/C(H)=C/\hspace{-0.3cm} \bigcap\limits_{\g\in Hom(C,H)}\hspace{-0.3cm} \ker(\g)\to \prod\limits_{\g\in Hom(C,H)} \hspace{-0.3cm} C/\ker(\g).\]
If $H$ is finite, then $\hom(C,H)$ is a finite set (since $C$ is finitely generated), hence $C/C(H)$ is finite.
If $H$ is furthermore solvable, then for any $\g\in \hom(C,H)$ the groups $C/\ker(\g)$ are solvable, hence $C/C(H)$ is solvable as well.
Moreover for any  $\g\in \hom(C,H)$ we have $\ell(C/\ker(\g))\leq \ell(H)$. We therefore get
\[ \ell(C/C(H))\leq \max_{\g\in Hom(C,H)}\ell(C/\ker(\g))\leq \ell(H).\]
\end{proof}

We will also need the following group homomorphism extension lemma.

\begin{lemma}\label{lem:extend}
Assume that $\statefs(n)$ holds and that $S$ is a finite solvable group with $\ell(S)\leq n$.
Let $\b:B\to S$ be a homomorphism.

Then there exists a $k\in \N$, a semidirect product $\Z/k\ltimes B/B(S)$ and a homomorphism
\[ \pi=\ll B,t | \i_-(A)=t\i_+(A)t^{-1}\rr  \to \Z/k\ltimes B/B(S) \]
which extends  $B\to B/B(S)$, i.e. we have the following commutative diagram:
\[ \xymatrix{
1\ar[r] & B/B(S)  \ar[r] & \Z/{k} \ltimes B/B(S)
\ar[r] & \Z/{k} \ar[r]& 1  \\
 & B \ar[u] \ar[r] & \pi.
\ar[u]  & &
 }\]
\end{lemma}

\begin{proof}
Assume that $\statefs(n)$ holds and that $S$ is a finite solvable group with $\ell(S)\leq n$.
Let $\b:B\to S$ be a homomorphism. We denote the projection map $B\to B/B(S)$ by $\rho$.

\begin{claim}
There exists an automorphism   $\g:B/B(S)\to B/B(S)$ such that
$\rho(\i_+(a))=\g(\rho(\i_-(a)))$ for all $a\in A$.
\end{claim}

Let $\i=\i_-$ or $\i=\i_+$.
By Lemma \ref{lem:chprop} we know that   $B/B(S)$ is finite solvable with $\ell(B/B(S))\leq n$.
It follows from  $\statefs(n)$ that
\[ \i_*:A/\ker\{A\xrightarrow{\i} B\xrightarrow{\rho} B/B(S)\} \to B/B(S) \]
is an isomorphism.
On the other hand it is also  a straightforward consequence of $\statefs(n)$ that
\[ \ker\{A\xrightarrow{\i} B\xrightarrow{\rho} B/B(S)\} =A(S).\]
Combining these two observations we see that $\i$ gives rise to an isomorphism $\i_*:A/A(S)\to B/B(S)$.
We now take $\gamma:=\i_{+*} \circ (\i_{-*})^{-1}$.
 This concludes the proof of the claim.

We now write $H=B/B(S)$. It is now straightforward to verify that
\[ \ba{rcl} \pi=\ll B,t | \i_-(A)=t\i_+(A)t^{-1}\rr &\to &   \Z\ltimes H=\ll H,t | H=t\g(H)t^{-1} \rr \\
 b&\mapsto & \rho(b), \,\,\, b\in B, \\
 t&\mapsto& t\ea \]
defines a homomorphism.
Since $H=B/B(S)$ is a finite group it follows that the
automorphism $\g$ has finite order $k$, in particular the projection map $\Z \ltimes
B/B(S)\to  \Z/k \ltimes B/B(S)$ is a homomorphism. Clearly the resulting homomorphism $\pi\to \Z/k\ltimes B/B(S)$ has all the required properties.
\end{proof}

We are in position now to prove Proposition \ref{prop:ind2}.

\begin{proof}[Proof~of~Proposition~\ref{prop:ind2}]
Recall that we assume that  $(N,\phi)$ satisfies Condition  $(*)$ and that  $\statefs(n)$ holds. We have to show that  $\stateh(n)$ holds as well.
So let $\b:B\to S$ be a homomorphism to a finite solvable group  $S$  with $\ell(S)\leq n$. We have to show  that for $\i=\i_{-},\i_+$
the homomorphism \[ \i_{*} :H_1(A;\Z[S])\to H_1(B;\Z[S]) \]
is an isomorphism. Without loss of generality we can assume that $\b$ is surjective. Recall that $\statefs(n)$ implies that
$\b\circ \i:A\to S$ is surjective as well.

We now apply Lemma \ref{lem:extend}  to find a
homomorphism
\[ \pi=\ll B,t | \i_-(A)=t\i_+(A)t^{-1}\rr  \to \Z/k\ltimes B/B(S) \]
which extends  $B\to B/B(S)$.
Note that
\be \label{equ:sameker} \ba{rcl} \ker\{\g:B\to \pi\to  \Z/k\ltimes B/B(S)\}&=&\ker\{B\to B/B(S)\}\\
\ker\{\g\circ \i:A\to B\to \pi\to \Z/k\ltimes B/B(S)\}&=&\ker\{\i:A\to B\to B/B(S)\}.\ea \ee
We let
\[ \ba{rclcl} \hat{B}&=&\ker\{B\to B/B(S)\},\\
\ti{B}&=&\ker(\b).\ea \]
Clearly $\hat{B}\subset \ti{B}$ by the definition of $B/B(S)$.
We also write $\hat{A}=\i^{-1}(\hat{B})$ and $\ti{A}=\i^{-1}(\ti{B})$.
We now consider the epimorphism $\pi_1(N)=\pi\to \Z/k\ltimes B/B(S)$.
By Condition ($*$), Equation (\ref{equ:sameker}), Proposition \ref{prop:sameh0h1} and  Corollaries \ref{cor:metabany} and  \ref{cor:metab} it follows that the maps
\[ \i:A/ \hat{A} \to B/\hat{B} \mbox{ and } \i:A/ [ \hat{A},\hat{A}  ]\to B/[\hat{B},\hat{B}]\]
are isomorphisms. It now follows from Lemma \ref{lem:alsoiso} and Corollary \ref{cor:metabany}
that the maps
\[  \i:A/ \ti{A} \to B/\ti{B} \mbox{ and } \i:A/ [ \ti{A},\ti{A}  ]\to B/[\ti{B},\ti{B}]\]
are also isomorphisms.
\end{proof}

\section{A product criterion}\label{section:product}
In this section we will apply a theorem of Agol to prove a criterion for a manifold to be a product, which complements
Proposition \ref{thm:finitesolvable}.

In order to state our result, we first recall the definition of a sutured manifold (cf. \cite[Definition~2.6]{Ga83} or \cite[p.~364]{CC03}).
A \emph{sutured manifold} $(M,\g)$ is a compact oriented 3--manifold $M$
together with a set $\g\subset \partial M$ of pairwise disjoint
annuli $A(\g)$ and tori $T(\g)$.
Furthermore, the structure of a sutured manifold consists of the following choices of orientations:
\bn
\item For each $A\in A(\g)$ a choice
of a simple closed, oriented curve in $A$ (called \emph{suture}) such that $A$ is the tubular neighborhood of the curve, and
\item the choice of an orientation for each component of $\partial M\sm A(\g)$.
\en
The orientations must be compatible, i.e.  the orientation of the components of $\partial M\sm A(\g)$ must be coherent with
the orientations of the sutures.

Given a sutured manifold $(M,\g)$ we define $R_+(\g)$
as the components of $\ol{\partial M \setminus \g}$ where the orientation agrees with
the orientation induced by $M$ on $\partial M$, and $R_-(\g)$ as the components of $\ol{\partial M \setminus \g}$
where the two orientations disagree. We define also $R(\g) = R_+(\g) \cup R_-(\g)$.

A sutured manifold $(M,\g)$ is called \emph{taut} if $M$ is irreducible and if each component of $R(\g)$
is incompressible and
Thurston norm--minimizing in $H_2(M,\g;\Z)$
(we refer to \cite{Sc89} for information regarding the Thurston norm on sutured manifolds).

An example of a taut sutured manifold is given by taking an oriented surface $\S$ and
considering $\S\times I$ with sutures given by the annuli $\partial \S\times I$. The sutures are oriented
by the orientation of $\partial \S$.   We can pick orientations such that  $R_-(\g)=\S\times 0$ and
$R_+(\g)=\S \times 1$.
If a sutured manifold $(M,\g)$ is diffeomorphic (as a sutured manifold) to such a product then we say that $(M,\g)$ is a \emph{product sutured manifold}.

Another example of a taut sutured manifold comes from considering an oriented incompressible Thurston norm minimizing surface $\S\subset N$
in an irreducible 3--manifold with empty or toroidal boundary.
We let $(M,\gamma)=(N\sm \nu \S,\partial N\cap (N\sm \nu \S))$.
With appropriate orientations  $(M,\gamma)$ is a taut sutured manifold
 such  that $R_-(\g)=\S^-$ and
$R_+(\g)=\S^+$.

The following theorem immediately implies
Theorem \ref{thm:agolintro}.

\begin{theorem} \label{thm:no}
Assume we have a taut sutured manifold $(M,\g)$ which has the following properties:
\bn
\item $R_\pm(\g)$ consist of one component  $\S^\pm$ each, and  the inclusion induced maps $\pi_1(\S^\pm)\to \pi_1(M)$
give rise to isomorphisms of the respective
prosolvable completions,
\item  $\pi_1(M)$ is residually finite solvable,
\en
then $(M,\g)$ is a product sutured manifold.
\end{theorem}

The key ingredient to the proof of Theorem \ref{thm:no} is a result of Agol's  \cite{Ag08}
 which we recall in  Section \ref{section:agol}.
We will then provide the proof for Theorem \ref{thm:no} in Sections \ref{section:proofno} and \ref{section:alexnorm}.

\begin{remark}
(1)  It is an immediate consequence of `peripheral subgroup separability' \cite{LN91} that the theorem holds
under the assumption that the inclusion induced maps $\pi_1(\S^\pm)\to \pi_1(M)$
give rise to isomorphisms of the respective
\emph{profinite} completions. It is not clear how the approach of \cite{LN91} can be adapted to prove
Theorem \ref{thm:no}.\\
(2) It is also interesting to compare Theorem \ref{thm:no} with a result of Grothendieck. In \cite[Section~3.1]{Gr70}
Grothendieck proves that if $\varphi:A\to B$ is a homomorphism between finitely presented, residually finite groups
which induces an isomorphism of the profinite completions, and if $A$ is arithmetic (e.g. a surface group), then $\varphi$
is an isomorphism. It is an interesting question whether Theorem \ref{thm:no} can be proved using purely group
theoretic arguments. We refer to \cite{AHKS07} for more information regarding this question.
\end{remark}

\subsection{Agol's theorem} \label{section:agol}
Before we can state Agol's result we have to introduce more definitions.
A group $G$ is called \emph{residually finite $\Q$--solvable} or \emph{RFRS} if there
exists a filtration  of groups $G=G_0\supset G_1 \supset G_2\dots $
such that the following hold:
\bn
\item $\cap_i G_i=\{1\}$,
\item   $G_i$ is a normal, finite index subgroup of  $G$ for any $i$,
\item for any $i$ the map $G_i\to G_i/G_{i+1}$ factors through $G_i\to H_1(G_i;\Z)/\mbox{torsion}$.
\en
Note that RFRS groups are in particular residually finite solvable, but the RFRS condition is considerably
stronger than being residually finite solvable. The notion of an RFRS group was introduced by Agol \cite{Ag08},
we refer to Agol's paper for more information on RFRS groups.

Given a sutured manifold $(M,\g)$ the double $DM_\g$ is defined to be the double of $M$ along $R(\g)$, i.e. $DM_\g=M\cup_{R(\g)}M$.
Note that the annuli $A(\g)$ give rise to toroidal boundary components of $DM_\g$.
We denote by $r:DM_\g\to M$ the retraction map given by `folding' the two copies of $M$ along $R(\g)$.

We are now in a position to state Agol's result.
The theorem as stated here is   clearly implicit in the proof of   \cite[Theorem~6.1]{Ag08}.

\begin{theorem} \label{thm:agol2}
Let $(M,\g)$ be a connected, taut sutured manifold
such that $\pi_1(M)$ satisfies property RFRS.
Write $W=DM_\g$.
Then there exists an epimorphism $\a:\pi_1(M)\to S$ to a finite solvable group,
such that in the covering $p:\widetilde{W}\to W$
corresponding to $\a\circ r_*:\pi_1(W)\to S$ the pull back of the class $[R_-(\g)]\in H_2(W,\partial W;\Z)$
lies on the closure of the cone over a fibered face
of  $\widetilde{W}$.
\end{theorem}

 Note that   $[R_+(\g)]=\pm [R_-(\g)]$  in $H_2(W,\partial W;\Z)$, i.e. $[R_-(\g)]$ is a fibered class if and only if $[R_+(\g)]$ is a fibered class.
In case that $\widetilde{W}$ has vanishing Thurston norm, then  we adopt the usual convention that by the fibered face we actually mean $H^1(\widetilde{W},\R)\sm \{0\}$.

\subsection{Proof of Theorem \ref{thm:no}}\label{section:proofno}

From now on assume we have a taut sutured manifold $(M,\g)$  with the following properties:
\bn
\item $R_\pm(\g)$ consist of one component  $\S^\pm$ each  and  the inclusion induced maps $\pi_1(\S^\pm)\to \pi_1(M)$  give rise to isomorphisms of the respective
prosolvable completions.
\item  $\pi_1(M)$ is residually finite solvable.
\en
Since Theorem \ref{thm:no} is obvious in the case $M=S^2\times [0,1]$ we will henceforth assume that $M\ne S^2\times [0,1]$.

Our main tool in proving Theorem \ref{thm:no} is Theorem \ref{thm:agol2}. In order to apply it we need the following claim.

\begin{claim} \label{lem:brfrs}
The group $\pi_1(M)$ is  RFRS.
\end{claim}

\begin{proof}
By assumption the group  $\pi_1(M)$ is residually finite solvable. This means that we can  find
a sequence $\pi_1(M)=B_0 \supset B_1 \supset B_2\dots $ with the following properties:
\bn
\item $\cap_i B_i=\{1\}$,
\item   $B_i$ is a normal, finite index  subgroup of  $\pi_1(M)$ for any $i$,
\item for any $i$ the map $B_i\to B_i/B_{i+1}$ factors through $B_i\to H_1(B_i;\Z)$.
\en
It remains to show that $B_i\to B_i/B_{i+1}$ factors through $H_1(B_i;\Z)/\mbox{torsion}$.
In fact we claim  that $H_1(B_i;\Z)$ is torsion--free.
Indeed, first note that by Shapiro's lemma $H_1(B_i;\Z)\cong H_1(B;\Z[B/B_i])\cong H_1(M;\Z[B/B_i])$.
Furthermore, by Lemma \ref{lem:twih1} we have
\[ H_1(\S^-;\Z[B/B_i]) \xrightarrow{\cong} H_1(M;\Z[B/B_i]), \]
but the first group is clearly torsion--free as it is the homology of a finite cover of a surface.
\end{proof}

In the following we write $W=DM_\g$.
By the above claim  we can apply Theorem \ref{thm:agol2} which says that
 there exists an epimorphism $\a:\pi_1(M)\to S$ to a finite solvable group, such that in the covering $p:\widetilde{W}\to W$
corresponding to $\a\circ r_*:\pi_1(W)\to S$ the pull back of the class $[R_-(\g)]=[\S^-]
\in H_2(W,\partial W;\Z)$
lies on the closure of the cone over   of a fibered face
of $\widetilde{W}$.

Note that we can view $\widetilde{W}$ as the double of the cover $(\widetilde{M},\ti{\g})$ of $(M,\g)$
induced by $\a:\pi_1(M)\to S$. We summarize the main properties of $\ti{\S}^\pm$ and $\widetilde{W}$
in the following lemma.

\begin{lemma}
\bn
\item  $\ti{\S}^\pm:=p^{-1}(\S^\pm)$ are connected surfaces,
\item the inclusion induced maps $\pi_1(\ti{\S}^\pm)\to \pi_1(\widetilde{M})$
give  rise to isomorphisms of
prosolvable completions,
\item  if
 $\ti{\S}^-$ is  the fiber of a fibration $\widetilde{W} = D\widetilde{M}_{\ti{\g}}  \to S^1$, then $\widetilde{M}$ is a product over $\ti{\S}^-$,
\item  $M$ is a product over $\S^-$ if and only if $\widetilde{M}$ is a product over $\ti{\S}^-$.
 \en
\end{lemma}

\begin{proof}
First note that it follows from Lemma \ref{lem:im} and Corollary \ref{cor:group1} and the assumption that   $\pi_1(\S^\pm)\to \pi_1(M)$
give rise to isomorphisms of the respective
prosolvable completions that $\pi_1(\S^\pm)\to \pi_1(M)\to S$ is surjective, i.e. the preimages $\ti{\S}^\pm:=p^{-1}(\S^\pm)$ are connected.
The second claim follows from Lemma \ref{lem:restrictiso} since the maps $\pi_1({\S}^\pm)\to \pi_1({M})$
give rise to isomorphisms of  their
prosolvable completions.

For the third claim consider the following commutative diagram
\[ \xymatrix{ \pi_1(\ti{\S}^-)\ar[dr]\ar[rr]&& \pi_1(\widetilde{W} \sm \nu \tilde{\Sigma}^-)\\
&\pi_1(\widetilde{M}).\ar[ur]&}\]
If  $\ti{\S}^-$ is  the fiber of a fibration $D\widetilde{M}_{\ti{\g}}\to S^1$, then the
 top map in the above commutative diagram is an isomorphism.
We can think of $\widetilde{W} \sm \nu \tilde{\Sigma}^-$ as $\widetilde{M}\cup_{\ti{\S}^+}\widetilde{M}$.
It is now clear that the lower two  maps are injective.
But then the lower left map also has to be an isomorphism, i.e. $\widetilde{M}$ is a product over $\ti{\S}^-$.
 The last claim is well--known, it is for example a consequence of \cite[Theorem~10.5]{He76}.
\end{proof}

Using the above lemma it is now clear that  the following lemma implies Theorem \ref{thm:no}.

\begin{lemma} \label{lem:no}
Let  $(M,\g)$ be a taut sutured manifold
 such that  $R_\pm(\g)$ consist of one component  $\S^\pm$ each.
Assume the following hold:
\bn
\item[(A)]  The inclusion induced maps $\pi_1(\S^\pm)\to \pi_1(M)$  give rise to isomorphisms of the respective
prosolvable completions.
\item[(B)] The class in $H^1(DM_\g;\Z)$ represented by the surface $ \S^-$
 lies on the closure of the cone over  a fibered face
of $DM_\g$.

\en
Then  $\S^-$ is the fiber of a fibration $DM_\g\to S^1$.
\end{lemma}

In the following we write $W=DM_\g$. Note that we have a canonical  involution $\tau$ on $W$ with fix point set $R(\g)$.
From now on we think of $W=DM_\g$ as $M\cup_{R(\g)} \tau(M)$.

Our main tool in the proof of Lemma \ref{lem:no} will be the interplay between the Thurston norm and
  McMullen's Alexander norm \cite{McM02}.
Recall that given a 3--manifold $V$ with $b_1(V)\geq 2$  the Alexander norm $||-||_A:H^1(V;\R)\to \R_{\geq 0}$ has the following properties:
\bn
\item[(a)] The Alexander norm ball is dual to the Newton polyhedron defined by the symmetrized Alexander polynomial $\Delta_V\in \Z[H_1(V;\Z)/\mbox{torsion}]$.
\item[(b)] The Alexander norm ball is a (possibly noncompact) polyhedron with finitely many faces.
\item[(c)] For any $\phi \in H^1(V;\R)$ we have $||\phi||_A\leq ||\phi||_T$, and equality holds for fibered classes.
\item[(d)] Let $C\subset H^1(V;\R)$ be a fibered cone, i.e. the cone on a fibered face of the Thurston norm ball, then $C$
is contained in the cone on the interior of a top--dimensional face of the Alexander norm ball.
\item[(e)] Let $C_1,C_2\subset H^1(V;\R)$ be fibered cones which are contained in the same cone  on the interior of a top--dimensional face of the Alexander norm ball,
then $C_1=C_2$.
\en

Our  assumption that the induced maps $\pi_1(\S^\pm)\to \pi_1(M)$  give rise to isomorphisms of the respective
prosolvable completions implies that $W=DM_\g$ `looks algebraically' the same as $\S^-\times S^1$.
More precisely, we have the following lemma which  we will prove in Section \ref{section:alexnorm}.

\begin{lemma} \label{lem:alexnorm}
Let  $(M,\g)$ be a taut sutured manifold
with the property that $R_\pm(\g)$ consist of one component  $\S^\pm$ each.
Assume that  (A) holds. Then the following hold:
\bn
\item There exists an isomorphism
\[ f:\R \oplus  H_{1}(\S^-,\partial \S^-;\R)\to H_2(W,\partial W;\R) \]
such that $f(1,0)=[\S^-]$ and such that $\tau(f(r,h))=f(r,-h)$.
\item The class $\phi=PD(\S^-)\in H^1(W;\Z)$  lies in the cone $D$ on the interior of a
top--dimensional face of the Alexander norm ball.
\en
\end{lemma}

Note that (1) implies in particular that $b_1(W)\geq 2$.
Assuming this lemma we are now in a position to prove Lemma \ref{lem:no}.

\begin{proof}[Proof of Lemma \ref{lem:no}]
Let  $(M,\g)$ be a taut sutured manifold
with the property that
$R_\pm(\g)$ consist of one component  $\S^\pm$ each.
Assume that (A) and (B) hold.

By Lemma \ref{lem:alexnorm} there exists a  cone $D\subset H^1(W;\R)$ on the interior of a
top--dimensional face of the Alexander norm ball which contains $\phi=PD([\S^-])$.
We denote the map
\[ \R \oplus  H_{1}(\S^-,\partial \S^-;\R)\xrightarrow{f} H_2(W,\partial W;\R) \xrightarrow{PD} H^1(W;\R) \]
by $\Phi$.

By (B) we can find  $h\in H_{1}(\S^-,\partial \S^-;\R)$ such that $\Phi(1,h)$ and $\Phi(1,-h)$ lie in $D$
and such that $\Phi(1,h)$ lies in the cone $C$ on a fibered face $F$ of the Thurston norm ball.
Note that the $\tau_*:H^1(W;\R)\to H^1(W;\R)$ sends fibered classes to fibered classes and preserves the Thurston norm. In particular
 $\tau(\Phi(1,h))=\Phi(1,-h)$ is fibered as well and it lies in  $\tau(C)$
which is the cone on the fibered face $\tau(F)$ of the Thurston norm ball.
Recall that $\tau(\Phi(1,h))=\Phi(1,-h)$ lies in $D$, it follows from Property (d) of the Alexander norm that $\tau(C)\subset D$.
We then use  (e) to conclude that $C=\tau(C)$. In particular $\Phi(1,h)$ and $\Phi(1,-h)$ lie in $C$.
Since $C$ is convex it follows that $\phi=\Phi(1,0)\in C$ i.e. $\phi$ is a fiber class.
\end{proof}

\subsection{Proof of Lemma  \ref{lem:alexnorm}}\label{section:alexnorm}

By Lemma \ref{lem:twih1} the following lemma is just the first statement of Lemma \ref{lem:alexnorm}.

\begin{lemma}
Let  $(M,\g)$ be a taut sutured manifold
with the property that $R_\pm(\g)$ consist of one component  $\S^\pm$ each.
Assume that  $\i_\pm: H_1(\S^\pm;\R)\to H_1(M;\R)$ are isomorphisms. Then there exists  an isomorphism
\[ f:\R \oplus  H_{1}(\S^-,\partial \S^-;\R)\to H_2(W,\partial W;\R) \]
such that $f(1,0)=[\S^-]$ and such that $\tau(f(b,c))=f(b,-c)$.
\end{lemma}

\begin{proof}
We start out with the following two claims.

\begin{claim} $M$ has no toroidal sutures.
\end{claim}

\begin{proof}
Denote the toroidal sutures by $T_1,\dots,T_n$.
Recall that for any compact 3--manifold $X$ we have $b_1(\partial X)\leq 2b_1(X)$.
In our case it is easy to see that we have $b_1(\partial M)=b_1(\S^-)+b_1(\S^+)+\sum_{i=1}^n b_1(T_i)=2b_1(\S)+2n$.
On the other hand, since  $H_1(\S^\pm;\R)\to H_1(M;\R)$ are isomorphisms we have $b_1(M)=b_1(\S)$.
It now follows from $b_1(\partial M)\leq 2b_1(M)$ that $n=0$.
\end{proof}

\begin{claim}
The inclusion induced maps $H_{1}(\S^\pm,\partial \S^\pm;\R)\to H_1(M,A(\g);\R)$
are isomorphisms.
\end{claim}

Now consider the following commutative diagram:
\[ \xymatrix{
H_1(\partial \S^-;\R)\ar[d]\ar[r] & H_1(\S^-;\R)\ar[d]^{\cong} \ar[r] & H_1(\S^-,\partial \S^-;\R)
\ar[d]\ar[r] & H_0(\partial \S^-;\R)\ar[r]\ar[d]& H_0(\S^-;\R)\ar[d]^\cong \\
H_1(A(\g);\R)\ar[r] & H_1(M;\R) \ar[r] & H_1(M,A(\g);\R)
\ar[r] & H_0(A(\g);\R)\ar[r]&H_0(M;\R).
 }\]
 Note that by the compatibility condition in the definition of sutured manifolds we have that
 for each component $A$ of $A(\g)$ the subset $\partial A\cap \S^-=A\cap \partial \S^-\subset \partial A$ consists of exactly
 one boundary component of $A$. This implies that  the maps $H_i(\partial \S^- \cap A;\R)\to H_i(A;\R)$ are isomorphisms.
The claim  now follows immediately from the above commutative diagram and from the assumption that  $H_1(\S^\pm;\R)\to H_1(M;\R)$ are isomorphisms.

 We now define
\[ g: H_{1}(\S^-,\partial \S^-;\R)\to H_{2}(W,\partial W;\R) \]
as follows: given an element $c\in H_{1}(\S^-,\partial \S^-;\R)$ represent it by a chain $c^-$,
since the maps $H_{1}(\S^\pm,\partial \S^\pm;\R)\to H_1(M,A(\g);\R)$ are isomorphisms
we can find a chain $c^+$ in $\S^+$ such that $[c^-]=[c^+]\in H_2(M,A(\g);\R)$.
Now let $d$ be a $2$--chain in $M$ such that $\partial d=c^-\cup -c^+$.
Then define $g(c)$ to be the element in $H_{2}(W,\partial W;\R)$ represented by the closed $2$--chain $d\cup -\tau(d)$.
It is easy to verify that $g$ is a well--defined homomorphism.
Note that $\partial W=A(\g)\cup \tau(A(\g))$ since $W$ has no toroidal sutures.
It is now  straightforward to check, using a Mayer--Vietoris sequence,
that the map
\[ \ba{rclcl} f:\R&\oplus& H_{1}(\S^-,\partial \S^-;\R)&\to& H_2(W,\partial W;\R) \\
(\, b&,& c\, )&\mapsto & b [\S^-]+g(c) \ea \]
is an isomorphism. Clearly $f(1,0)=[\S^-]$.
It is also easy to verify  that $\tau(f(b,c))=f(b,-c)$. This shows that  $f$ has all the required properties.
\end{proof}

The second statement of Lemma  \ref{lem:alexnorm} is more intricate.
We start with the following lemma which in light of \cite{Gr70}, \cite{BG04} and \cite{AHKS07} has perhaps some independent interest.

\begin{lemma} \label{lem:isometab}
Let $\varphi:A\to B$ be a homomorphism of finitely generated metabelian groups which induces an isomorphism of prosolvable completions.
Then $\varphi$ is also an isomorphism.
\end{lemma}

\begin{proof}

We first show that $A\to B$  is an injection.
We consider the following commutative diagram
\[ \xymatrix{ A\ar[d]\ar[r] & B\ar[d] \\  \widehat{A}_{\fs}\ar[r] & \widehat{B}_{\fs}.}\]
The vertical maps are injections since metabelian groups are residually finite (cf. \cite{Ha59}).
The bottom map is an isomorphism by assumption.
It now follows that the top map is an injection.

Now suppose that the homomorphism $A\to B$ is not surjective.
 We  identify $H_1(A;\Z)=H_1(A;\Z)\xrightarrow{\cong} H_1(B;\Z)=H_1(B;\Z)$ via $\varphi$ and refer to the group
as $H$. Let $g'\in B\sm \varphi(A)$.
We can pick an $a\in A$ such that $\varphi(a)$ and $g'$ represent the same element in $H$.
Let $g=\varphi(a)^{-1}g'$. Then $g$ represents the trivial element in $H$ but $g$ is also an element in
$B\sm \varphi(A)$.

We will show that there exists a  homomorphism
$\a:B\to G$ to a finite metabelian group such that $\a$ separates $g$ from $\varphi(A)$, i.e. such that $\a(g)\not\in \a(\varphi(A))$.
This  then immediately contradicts, via Lemma \ref{lem:pro}, our assumption that
$\varphi:A\to B$  induces an isomorphism of prosolvable completions.
Our construction of finding such $\a$ builds on some ideas of the proof of \cite[Theorem~1]{LN91}.

We write ${B_1}=B_2={B}$.  We denote the inclusion maps $A\to B_i=B$ by $\varphi_i$.
We let $C=B_1*_{A} B_2$.
It is straightforward to see that the homomorphisms $B_i \to C$ give rise
to an isomorphism $H_1(B_i;\Z)=H\to H_1(C;\Z)$.
Denote by
\[ s:B_1*_{A} B_2\to B_1*_{A} B_2\] the
switching map, i.e. the map induced by $s(b)=b\in B_2$ for $b\in B_1$ and $s(b)=b\in
B_1$ for $b\in B_2$. Note that $s$ acts as the identity on
$A\subset C$. Also note that $s$ descends to a map $C/C^{(2)}\to C/C^{(2)}$ which we also denote
by $s$.
We now view $g$ as an element in $B_1$ and hence as an element in $C$.
Note that the fact that $g$ represents the trivial element in $H$ implies that $g$ represents an element
in $C^{(1)}/C^{(2)}$.
 We will
first show that $s(g)\ne g \in C/C^{(2)}$. Consider the following commutative diagram of exact sequences
\[ \ba{ccccccccccccccc}
H_1(A;\Z[H])&\to&H_1(B_1;\Z[H])\oplus H_1(B_2;\Z[H])&\to&
H_1(C;\Z[H])&\to&0\\[2mm]
\,\,\, \downarrow \cong &&\, \,\,\downarrow \cong &&\,\,\downarrow \cong &&\downarrow &&\\[2mm]
 A^{(1)}&\to&B_1^{(1)}\times B_2^{(1)}&\to&C^{(1)}/C^{(2)}&\to&0.\\
 h&\mapsto & (\varphi_1(h),\varphi_2(h)^{-1})&\ea
 \]
 Since $g\in B_1^{(1)}\sm \varphi(A^{(1)})$ it follows
that $(g,g^{-1})$ does not lie in the image of $ A^{(1)}$ in $B_1^{(1)}\times
B_2^{(1)}$. It therefore follows from the above diagram that $gs(g)^{-1}\ne e \in
C^{(1)}/C^{(2)}$.

Note that  $C/C^{(2)}$ is metabelian, and hence by \cite{Ha59}
residually finite. We can therefore find an epimorphism $\a: C/C^{(2)} \to
G$ onto a finite group $G$ (which is necessarily metabelian) such that $\a(gs(g)^{-1})\ne e$. Now consider $\b:C/C^{(2)}\to G\times
G$ given by $\b(h)=(\a(h),\a(s(h)))$. Then clearly $\b(g)\not\in \b(A)\subset \{ (g,g) \,|\, g\in G\}$.
The restriction of $\b$ to $B=B_1$ now clearly separates $g$ from $A$.
\end{proof}

\begin{corollary} \label{cor:isometab}
Let $\varphi:A\to B$ be a homomorphism of finitely generated groups which induces an isomorphism of prosolvable completions.
Then the induced map $A/A^{(2)}\to B/B^{(2)}$ is an isomorphism.
\end{corollary}

\begin{proof}
It follows immediately from Lemma \ref{lem:pro} that  $\varphi$ induces an isomorphism of the prosolvable completions of the metabelian groups $A/A^{(2)}$ and $B/B^{(2)}$.
It now follows from Lemma \ref{lem:isometab} that the induced map $A/A^{(2)}\to B/B^{(2)}$ is an isomorphism.
\end{proof}

We now turn to the proof of the second claim of Lemma \ref{lem:alexnorm}.
For the remainder of this section let  $(M,\g)$ be a taut sutured manifold
with the property that $R_\pm(\g)$ consist of one component  $\S^\pm$ each.
We assume that  (A) holds, i.e. the inclusion induced maps $\pi_1(\S^\pm)\to \pi_1(M)$  give rise to isomorphisms of the respective
prosolvable completions. We have to show that the class $\phi=PD(\S^-)\in H^1(W;\Z)$  lies in the cone on the interior of a
top--dimensional face of the Alexander norm ball.

For the remainder of the section we pick a base point $x^-\in \S^-$ and a base point $x^+\in \S^+$.
We endow $W, M$ and $\tau(M)$ with the base point $x^-$. Furthermore we pick a path $\g$ in $M$ connecting $x^-\in \S^-$ to  $x^+\in \S^+$.

Our goal is to understand the Alexander norm ball of $W$.
In order to do this we first have to study $H=H_1(W;\Z)$.
Let $t$ denote the element in $H$ represented by the closed path $\g \cup -\tau(\g)$.
It follows from a straightforward Mayer--Vietoris sequence argument  that  we have an isomorphism
\[ \ba{rclcl} f:H_1(\S^-;\Z)&\oplus& \ll t\rr &\to& H_1(W;\Z) \\
(\, b&,& t^k\, )&\mapsto & \i(b)+kt. \ea \]
In particular $H$ is torsion--free.
We write $F=H_1(\S^-;\Z)$. We  use $f$ to identify $H$ with $F\times \ll t\rr$ and to identify $\Z[H]$ with $\Z[F]\tpm$.

We now consider the Alexander module $H_1(W;\Z[H])$. Recall that $H_1(W;\Z[H])$ is the homology of the
covering of $W$ corresponding to $\pi_1(W)\to H_1(W;\Z)=H$ together with the $\Z[H]$--module
structure given by deck transformations.

In the following claim we compare $W$ with $\S\times S^1$. We also write $F=H_1(\S;\Z)$ and we can identify
$H_1(\S\times S^1;\Z)$ with $H=F\times \ll t\rr$. In particular we identify $H_1(\S\times S^1;\Z)$ with $H_1(W;\Z)$.
With these identifications we can now state the following lemma.

\begin{lemma}
The $\Z[H]$--module $H_1(W;\Z[H])$ is isomorphic to the $\Z[H]$--module $H_1(\S\times S^1;\Z[H])$.
\end{lemma}

\begin{proof}
In the following we identify $\S$ with $\S^-\subset W$.
We denote by $X$ the result of gluing $M$ and $\tau(M)$ along $\S=\S^-$.
Note that we have two canonical maps $r:\S^+\to M\to X$ and $s:\S^+\to \tau(M)\to X$.
We furthermore denote the canonical inclusion maps $\S\to M, \S\to \tau(M)$ and $\S=\S^-\to X$  by $i$.
Throughout this proof we denote by $i,r,s$ the induced maps on solvable quotients as well.
\\

\noindent \emph{Claim A.}
The map $i:\pi_1(\S)\to \pi_1(X)$ gives rise to an isomorphism $\pi_1(\S)/\pi_1(\S)^{(2)}\to \pi_1(X)/\pi_1(X)^{(2)}$.
\\

In the following let $M'$ be either $M$ or $\tau(M)$.
Recall that we assume that $\pi_1(\S)\to \pi_1(M')$ gives rise to isomorphisms of the prosolvable completions.
It now  follows from Corollary \ref{cor:isometab} that $\pi_1(\S)/\pi_1(\S)^{(2)}\to \pi_1(M'/\pi_1(M')^{(2)}$ is an isomorphism.
Now let $g:\pi_1(X)=\pi_1(M\cup_{\S} \tau(M))\to \pi_1(M))$ be the `folding map'.
Note that
\[ \pi_1(\S)/\pi_1(\S)^{(2)}\xrightarrow{i}\pi_1(X)/\pi_1(X)^{(2)}\xrightarrow{g}\pi_1(M)/\pi_1(M)^{(2)}\xleftarrow{\cong}
\pi_1(\S)/\pi_1(\S)^{(2)}\]
is the identity map. In particular $\pi_1(\S)/\pi_1(\S)^{(2)}\xrightarrow{i}\pi_1(X)/\pi_1(X)^{(2)}$ is injective.
On the other hand it follows from the van Kampen theorem that
\[ \pi_1(X)=\pi_1(M)*_{\pi_1(\S)}\pi_1(M'), \]
in particular $\pi_1(X)/\pi_1(X)^{(2)}$ is generated by the images of $\pi_1(M)$ and $\pi_1(\tau(M))$ in $\pi_1(X)/\pi_1(X)^{(2)}$.
But it follows immediately from the following commutative diagram
\[ \xymatrix{ \pi_1(\S) \ar[rrrr] \ar[dr]\ar@/_3pc/[ddrr] &&&& \pi_1(M') \ar[dl] \ar@/^3pc/[ddll] \\
&\pi_1(\S)/\pi_1(\S)^{(2)}\ar[rr]^{\cong}\ar[dr] && \pi_1(M')/\pi_1(M')^{(2)} \ar[dl]& \\
&& \pi_1(X)/\pi_1(X)^{(2)} &&}\]
that  image of $\pi_1(\S)/\pi_1(\S)^{(2)}$ in $\pi_1(X)/\pi_1(X)^{(2)}$ also generates the group.
This concludes the proof of the claim A.\\

\noindent \emph{Claim B.}
For any $g\in \pi_1(\S^+)/\pi_1(\S^+)$ we have
\[ r(g)=s(g) \in \pi_1(X)/\pi_1(X)^{(2)}\]

Denote by $\tau:X=M\cup_\S \tau(M) \to X=M\cup_\S \tau(M)$ the map given by switching the two copies of $M$. Clearly $r(g)=\tau_*(s(g))$. But $\tau_*$ acts trivially on image of $\pi_1(\S)/\pi_1(\S)^{(2)}$
in $\pi_1(X)/\pi_1(X)^{(2)}$. By the above claim this means that $\tau_*$ acts trivially on  $\pi_1(X)/\pi_1(X)^{(2)}$. This concludes the proof of the claim.

We now view $W$ as the result of gluing the two  copies of $\S^+$ in $\partial X$ by the identity map.
First note that by the van Kampen theorem we have
\[ \pi_1(W)=\ll t,\pi_1(X) \, |\, ts(g)t^{-1}=r(g), g\in \pi_1(\S^+)\rr. \]
Note that by Claim B the obvious assignments give rise to a well--defined map
\[ \pi_1(W)=\ll t,\pi_1(X) \, |\, ts(g)t^{-1}=r(g), g\in \pi_1(\S^+)\rr \to  \ll t \rr \times  \pi_1(X)/\pi_1(X)^{(2)}. \]
Since $\pi_1(X)/\pi_1(X)^{(2)} $ is metabelian this map  descends to a map
\[ \Phi:  \ll t,\pi_1(X) \, |\, ts(g)t^{-1}=r(g), g\in \pi_1(\S^+)\rr/(\dots)^{(2)}\to \ll t \rr \times \pi_1(X)/\pi_1(X)^{(2)}. \]

\noindent \emph{Claim C.}
The map $\Phi: \pi_1(W)/\pi_1(W)^{(2)} \to \ll t \rr \times \pi_1(X)/\pi_1(X)^{(2)}$ is an isomorphism.
\\

We denote by $\Psi$ the following map:
\[ \ba{rcl} \ll t \rr  \times \pi_1(X)/\pi_1(X)^{(2)}
&\to &\ll t,\pi_1(X)/\pi_1(X)^{(2)} \, |\, ts(g)t^{-1}=r(g), g\in \pi_1(\S^+)\rr\\[1mm]
&=&\ll t,\pi_1(X) \, |\, ts(g)t^{-1}=r(g), g\in \pi_1(\S^+), \pi_1(X)^{(2)}\rr\\[1mm]
&\to &  \ll t,\pi_1(X) \, |\, ts(g)t^{-1}=r(g), g\in \pi_1(\S^+)\rr/(\dots)^{(2)}.\ea \]
Clearly $\Psi$ is surjective and we have $\Phi\circ \Psi=\id$. It follows that $\Phi$ is an isomorphism. This concludes the proof of the claim.

Finally note that we have a canonical isomorphism
\[ \pi_1(\S\times S^1)/\pi_1(\S\times S^1)^{(2)} = \ll t\rr \times \pi_1(\S)/\pi_1(\S)^{(2)}.\]
It now follows
\ from the above discussion that we have an isomorphism
\[ \ba{rcl} \pi_1(W)/\pi_1(W)^{(2)}&\xrightarrow{\Phi}& \ll t\rr \times \pi_1(X)/\pi_1(X)^{(2)} \\
&\cong& \ll t\rr \times \pi_1(\S)/\pi_1(\S)^{(2)} \\
&=&\pi_1(\S\times S^1)/\pi_1(\S\times S^1)^{(2)}\ea \]
which we again denote by $\Phi$. Note that under the abelianization the map $\Phi$ descends to the above identification $H_1(\S\times S^1;\Z) =H=H_1(W;\Z)$.
We now get the following commutative diagram of exact sequences
\[ \xymatrix{  0\ar[r] &H_1(W;\Z[H])\ar[d]\ar[r]&\pi_1(W)/\pi_1(W)^{(2)}
\ar[d]^{\Phi}\ar[r] &H:=H_1(W;\Z)\ar[d]^{=} \ar[r]&0 \\
0\ar[r] &H_1(\S\times S^1;\Z[H])\ar[r]&\pi_1(\S\times S^1)/\pi_1(\S\times S^1)^{(2)}
\ar[r] &H_1(\S\times S^1;\Z)  \ar[r]&0.}\]
The lemma is now immediate.
\end{proof}

We are now ready to  prove the second statement of Lemma \ref{lem:alexnorm}.
Note that the isomorphism of Alexander modules implies that the Alexander polynomials $\Delta_W$ and $\Delta_{\S\times S^1}$ agree in $\Z[H]$.
It is well--known that $\Delta_{\S\times S^1}=(t-1)^{||\phi||_T}\in \Z[H]=\Z[F]\tpm$.
Recall that we are interested in $\phi=PD([\S])$, and that $\phi$ as an element in
$\hom(H,\Z)=\hom(F\times \ll t\rr),\Z)$ is given by $\phi(t)=1, \phi|_F=0$. It is now obvious from $\Delta_W=\Delta_{\S\times S^1}=(t-1)^{||\phi||_T}$
that $\phi$ lies
in the interior of a top--dimensional face of the Alexander norm ball of $W$.
This concludes the proof of the second statement of Lemma \ref{lem:alexnorm}
modulo the proof of the claim.

\section{Residual properties of 3--manifold groups}\label{section:virtp}

Proposition  \ref{thm:finitesolvable} and Theorem \ref{thm:no} are almost enough to deduce Theorem \ref{mainthm},
but we still have to deal with the assumption in Theorem \ref{thm:no} that $\pi_1(W)$ has to be residually finite solvable.

Using well--known arguments (see Section \ref{section:mainthm} for details) one can easily see that Proposition  \ref{thm:finitesolvable} and Theorem \ref{thm:no}
imply Theorem \ref{mainthm} for 3--manifolds $N$ which have virtually residually finite solvable fundamental groups.
Here we say that a group $\pi$ has virtually a property if a finite index subgroup of $\pi$ has this property.
It seems reasonable to conjecture that all 3--manifold groups are virtually residually finite solvable.
For example linear groups (and hence
fundamental groups of hyperbolic 3--manifolds and Seifert fibered spaces) are virtually residually finite solvable and (virtually) fibered 3--manifold groups are easily seen to be (virtually) residually finite solvable.

It is not known though whether all 3--manifold groups are linear.
In the case of 3--manifolds with non--trivial JSJ decomposition
we therefore
use a slightly different route to deduce Theorem \ref{mainthm} from
Proposition  \ref{thm:finitesolvable} and Theorem \ref{thm:no}.
In Lemmas \ref{lem:prime} and \ref{lem:closed} we first show that it suffices in the proof of Theorem \ref{mainthm} to consider closed prime 3--manifolds. In this  section we will show that given
a closed prime  3--manifold $N$,  there exists a finite cover $N'$
 of $N$ such that  all pieces of the JSJ decomposition of $N'$
have residually finite solvable fundamental groups (Theorem \ref{thm:virtp}). Finally in Section
\ref{section:jsj} we will prove  a result which allows us in the proof of Theorem \ref{mainthm} to work with each
JSJ piece separately (Theorem \ref{thm:miiso}).

\subsection{Statement of the theorem}

We first recall some definitions.
Let $p$ be a prime. A $p$--group is a group such that the order of the group is a power of $p$.
Note that any $p$--group is  in particular finite solvable.
A group $\pi$ is called \emph{residually $p$} if for any nontrivial $g\in \pi$ there exists a homomorphism $\a:\pi\to P$ to a $p$--group such that $\a(g)\ne e$.
A residually $p$ group is evidently also residually finite solvable.

For the reader's convenience we recall the statement of Theorem \ref{thm:virtpintro} which we will prove in this section.

\begin{theorem} \label{thm:virtp}
Let $N$ be a closed irreducible 3--manifold.
Then for all but finitely many primes $p$ there exists a finite cover $N'$ of $N$
such that the fundamental group of any JSJ component of $N'$ is residually $p$.
\end{theorem}

\begin{remark}
\bn
\item Note that this theorem relies on the geometrization
results of Thurston and Perelman.
\item A slight modification of our proof shows that the statement of the theorem also holds for irreducible 3--manifolds with toroidal boundary.
\en
\end{remark}

\subsection{Proof of Theorem \ref{thm:virtp}}

The proof of the theorem combines in a straightforward way ideas from the proof that
finitely generated subgroups of $\gl(n,\C)$
are virtually residually $p$ for all but finitely many primes $p$
(cf. e.g. \cite[Theorem~4.7]{We73} or \cite[Window~7,~Proposition~9]{LS03})
with  ideas from  the proof that 3--manifold groups are residually finite (cf.  \cite{He87}).
Since all technical results can be found in either \cite{We73} or \cite{He87}, and in order to save space,
we only give an outline of the proof by referring heavily to \cite{We73} and \cite{He87}.

In the following recall that given a positive integer $n$ there exists a unique characteristic subgroup of $\Z\oplus \Z=\pi_1(\text{torus})$ of index $n^2$, namely $n(\Z\oplus \Z)$.

\begin{defn} \label{defn:filtration}
Let $N$ be a  $3$-manifold which is either closed or has  toroidal boundary.
Given a prime $p$ we say that a subgroup $\G\subset \pi_1(N)$ has Property $(p)$
if it satisfies the following two conditions:
\bn
\item $\G$ is residually $p$, and
\item for any torus $T\subset \partial N$ the group $\G\cap \pi_1(T)$ is the unique characteristic subgroup of $\pi_1(T)$ of index $p^{2}$.
\en
\end{defn}

\begin{proposition} \label{prop:virtphyp}
Let $N$ be a compact orientable $3$-manifold with empty or toroidal
 boundary such that the interior  has a complete hyperbolic structure of finite volume.  Then for all but finitely many primes $p$ there exists a finite index subgroup of  $\pi_1(N)$ which has Property $(p)$.
\end{proposition}

\begin{proof}
This theorem is essentially a straightforward combination of \cite[Lemma~4.1]{He87} with the proof that finitely generated
linear groups are virtually residually $p$.  We will use throughout  the notation of the proof of \cite[Lemma~4.1]{He87}. First we pick a finitely generated subring $A\subset \C$ as in \cite[Proof~of~Lemma~4.1]{He87}.
In particular we can assume that  $\pi_1(N)\subset \mbox{SL}(2,A)$ where $A\subset \C$.
We pick a prime $p$  and a maximal ideal $\mm\subset A$ as in \cite[p.~391]{He87}. We then have in particular
that $\mbox{char}(A/\mm)=p$. For $i\geq 1$ we now let $\G_i=\ker\{\pi_1(N)\to \mbox{SL}(2,A/\mm^i)\times H/p^iH\}$ where $H=H_1(N;\Z)/\mbox{torsion}$. We claim that $\G_1\subset \pi_1(N)$ is a finite index subgroup which has Property $(p)$.
Clearly  $\G_1$ is of finite index in $\pi_1(N)$ and
by \cite[p.~391]{He87} the subgroup $\G_1$ also satisfies condition  (2). The proof
that finitely generated linear groups are virtually residually $p$ (cf. \cite[Theorem~4.7]{We73} or \cite[Window~7,~Proposition~9]{LS03}) then shows immediately that all the groups $\G_1/\G_i, i\geq 1$ are $p$--groups and that $\cap_{i=1}^{\infty} \G_i=\{1\}$. In particular $\G_1$ is  residually $p$.
\end{proof}

\begin{proposition} \label{prop:virtpseifert}
Let $N$ be a Seifert fibered space.
Then for all but finitely many primes $p$, there exists a finite index subgroup of  $\pi_1(N)$ which has Property $(p)$.
\end{proposition}

\begin{proof}
If $N$ is a closed Seifert fibered space, then it is well--known that $\pi_1(N)$ is linear,
and the proposition immediately follows from the fact that linear groups are virtually residually $p$
for almost all primes $p$.

Now consider the case that $N$ has boundary.
It is well--known (cf. for example \cite[Lemma~6]{Ha01} and see also \cite[p.~391]{He87}) that there exists a finite cover $q:N'\to N$ with the following two properties:
\bn
\item $N'=S^1\times F$ for some surface $F$,
\item for any torus $T\subset \partial N$ the group $\pi_1(N') \cap \pi_1(T)$ is the unique characteristic subgroup of $\pi_1(T)$ of index $p^{2}$.
\en
We now write
$\G:= \pi_1(N')\subset \pi_1(N)$.
The group $\G$ is residually $p$ since
 free groups are residually $p$. It now follows from (2) that $\G$ has the required properties.
\end{proof}

We are now in a position to prove Theorem \ref{thm:virtp}.

\begin{proof}[Proof of Theorem \ref{thm:virtp}]
Let $N$ be a closed irreducible 3--manifold.
Let $N_1,\dots,N_r$ be the JSJ components. For all but finitely many primes $p$
we can by Propositions  \ref{prop:virtphyp} and \ref{prop:virtpseifert} find finite index subgroups
$\G_i\subset \pi_1(N_i)$ for $i=1,\dots,r$ which have Property $(p)$. We denote by $N_i'$ the cover of $N_i$ corresponding
to $\G_i$.

By the second condition of Property $(p)$ the intersections of the subgroups $\G_i, i=1,\dots,r$ with the fundamental group of  any
torus of the JSJ decomposition coincide. We can therefore appeal to  \cite[Theorem~2.2]{He87} to
find a finite cover $N'$ of $N$ such that any component in the JSJ decomposition
of $N'$ is homeomorphic to some $N_i', i\in \{1,\dots,r\}$. Recall that $\pi_1(N_i')=\G_i$ is residually $p$
for any $i$, hence the cover
$N'$ of $N$ has the desired properties.
\end{proof}

\section{The JSJ decomposition and prosolvable completions} \label{section:jsj}

Let $N$ be a closed 3--manifold and let $\phi\in H^1(N;\Z)$ primitive with $||\phi||_T>0$.
If $(N,\phi)$ fibers, and if $\S\subset N$ is a  surface dual to $\phi$ which is the fiber of the fibration, then
it is well--known (cf. e.g. \cite{EN85}) that the JSJ tori of $N$ cut the product $N\sm \nu \S\cong \S\times [0,1]$ into smaller products.

If $(N,\phi)$ satisfies Condition ($*$), and if  $\S\subset N$ is a connected Thurston norm minimizing surface dual to $\phi$,
then we will see in Lemma \ref{lem:aiconnect} and Theorem \ref{thm:miiso}
that the JSJ tori of $N$ cut the manifold $N\sm \nu \S$ into smaller pieces which look like products
 `on the level of prosolvable completions'.
This result  will play an important role in the proof of Theorem \ref{mainthm} as it allows us to  work with each
JSJ piece separately.

\subsection{The statement of the theorem}\label{section:deti}

Throughout this section let  $N$ be a closed
irreducible 3--manifold.
Furthermore let  $\phi \in H^{1}(N;\Z)$ be a primitive class which is dual to a \emph{connected} Thurston norm minimizing surface. (Recall that by Proposition \ref{prop:mcm} this is in particular the case if $(N,\phi)$ satisfies Condition  $(*)$.) Finally we assume that $||\phi||_T>0$.

We now fix once and for all embedded tori  $T_1,\dots,T_r\subset N$  which give the JSJ decomposition of $N$.
(Recall that the $T_1,\dots,T_r$ are unique up to reordering and isotopy.)

We will make several times use of the following well--known observations:
\begin{lemma}\label{lem:nontrivialonsigma}
Let $\S\subset N$ be an incompressible surface in general position with the JSJ torus $T_i, i\in 1,\dots,r$.
Let $c$ be a component of $ \S\cap T_i$. Then $c$ represents a non--trivial element in $\pi_1(T_i)$ if and only
if $c$ represents a non--trivial element in $\pi_1(\S)$.
\end{lemma}

\begin{lemma}\label{lem:minimalsigma}
There exists an embedded Thurston norm minimizing surface $\S\subset N$ dual to $\phi$ with the following three properties:
\bn
\item $\S$ is connected,
\item the tori $T_i, i=1,\dots,r$ and the surface $\S$ are in general position,
\item any component of $\S\cap T_i, i=1,\dots,r$
represents a nontrivial element in $\pi_1(T_i)$.
\en
\end{lemma}

Now, among all surfaces dual to $\phi$ satisfying the properties of the lemma we pick a surface $\S$
which minimizes the number $\sum_{i=1}^r b_0(\S\cap T_i)$.

Given $\S$ we can and will fix a tubular neighborhood $\S\times [-1,1]\subset N$
such that the tori $T_i, i=1,\dots,r$ and the surface $\S\times t$ are in general position
for any $t\in [-1,1]$.
We from now on  write
$M=N\sm \S\times (-1,1)$ and  $\S^\pm=\S\times \pm 1$.

We denote the components of $N$ cut along
$T_1,\dots,T_r$ by $N_1,\dots,N_s$.
Let $\{A_1,\dots,A_m\}$ be the set of components of the intersection of the tori $T_1\cup \dots \cup T_r$ with $M$.
Note that the surfaces $A_i\subset M, i=1,\dots,m$ are properly embedded since we assumed that the tori $T_i$ and the surfaces
$\S^\pm=\S\times \pm 1$ are in general position.
We also  let $\{M_1,\dots,M_n\}$ be the set of components of the intersection of $N_i$ with $M$ for $i=1,\dots,s$.
Put differently, $M_1,\dots, M_n$ are the components of $M$ cut along $A_1,\dots,A_m$.
For $i=1,\dots,n$ we furthermore write $\S_i^\pm=M_i\cap \S^\pm$.

Let  $i\in \{1,\dots,m\}$.  If the surface $A_i$ is an annulus, then we say that $A_i$ \emph{connects $\S^-$ and $\S^+$} if
 one boundary component of $A_i$ lies on $\S^-$ and the other boundary component lies on $\S^+$.
 The following lemma will be proved in Section \ref{section:annuli}

\begin{lemma} \label{lem:aiconnect}
Assume that $(N,\phi)$ satisfies Condition  $(*)$, then for $i=1,\dots,m$ the surface $A_i$ is an annulus which connects $\S^-$ and $\S^+$.
\end{lemma}

We can now formulate the main theorem of this section. The proof  will be given
in Sections \ref{section:annuli}, \ref{section:guts} and \ref{section:hyps}.

\begin{theorem}\label{thm:miiso}\label{thm:jsjsep}
Assume that for $i=1,\dots,m$ the surface $A_i$ is an annulus which connects $\S^-$ and $\S^+$.
Furthermore assume that the inclusion induced maps $\pi_1(\S^\pm)\to \pi_1(M)$ give rise to an isomorphism
of prosolvable completions.  Then for $i=1,\dots,n$ the following hold:
\bn
\item
The surfaces $\S_i^\pm$ are connected.
\item Given $j\in \{1,\dots,n\}$ with $M_i\subset N_j$ the inclusion induced map $\pi_1(M_i)\to \pi_1(N_j)$ is injective.
\item The inclusion induced maps $\pi_1(\S_i^\pm)\to \pi_1(M_i)$ give rise to  isomorphisms
of the respective prosolvable completions.
\en
\end{theorem}

We would like to remind the reader that at the beginning of the section we made the assumption that $||\phi||_T>0$.

\subsection{Proof of Lemma \ref{lem:aiconnect}}\label{section:annuli}

We first recall  the following theorem from an earlier paper (cf. \cite[Theorem~5.2]{FV08b}).

\begin{theorem} \label{thm:fv08torus}Let $Y$ be a closed irreducible 3--manifold.
Let  $\psi \in H^1(Y;\Z)$ a primitive
class. Assume that for any epimorphism $\a:\pi_1(Y)\to G$ onto a finite group $G$
the twisted Alexander polynomial $\Delta_{Y,\psi}^\a\in \zt$  is nonzero. Let $T \subset Y$ be an
incompressible embedded torus. Then either $\psi|_T\in  H^1(T;\Z)$ is nonzero, or $(Y,\psi)$
fibers over $S^1$ with fiber $T$.
\end{theorem}

With this theorem we are now able to prove Lemma \ref{lem:aiconnect}.
We use the notation from the previous section.
So assume that $(N,\phi)$ is a pair which satisfies Condition  $(*)$.
In particular we have that $\Delta_{N,\phi}^\a \ne 0$ for any epimorphism $\a:\pi_1(N)\to G$ onto a finite group $G$.
We can therefore apply Theorem \ref{thm:fv08torus} to the tori $T_1,\dots,T_r\subset N$
to conclude that either $(N,\phi)$
fibers over $S^1$ with toroidal fiber, or  $\phi|_{T_i}\in  H^1(T_i;\Z)$ is nonzero for $i=1,\dots,r$.
Recall that we assumed that $||\phi||_T>0$,  we therefore only have to deal with the
 latter case. From   $\phi|_{T_i}\in  H^1(T_i;\Z)$ nonzero we obtain that  $\S$ (which is dual to $\phi$)  necessarily intersects
$T_i$ in at least one curve which is homologically essential on $T_i$.
 In fact by our assumption on $\S$ and $T_1,\dots,T_r$ any intersection curve $\S\cap T_i\subset T_i$ is essential, in particular the components of $T_i$ cut along $\S$ are indeed annuli.

In order to prove Lemma  \ref{lem:aiconnect} it now remains to show that
each of the  annuli $A_i$  connects $\S^-$ and $\S^+$.
So assume there exists an $i\in \{1,\dots,m\}$ such that  the annulus $A_i$ does not connect $\S^-$ and $\S^+$.
Without loss of generality we can assume  that $\S^+\cap A_i=\emptyset$.
We equip $A_i$ with an orientation.
Denote the two oriented components
of $\partial A_i$ by $c$ and $-d$. By our assumption $c$ and $d$ lie in $\S^-$, and they cobound the annulus
$A_i\subset M$.

Now recall that by Proposition \ref{prop:sameh0h1} our assumption that
$(N,\phi)$ satisfies Condition  $(*)$ implies in particular that
 $H_1(\S^-;\Z)\to H_1(M;\Z)$ is an isomorphism.
Note that $c,d$ are homologous in $M$ via the annulus $A:=A_i$, and since $H_1(\S^-;\Z)\to H_1(M;\Z)$ is an isomorphism we deduce that
$c$ and $d$ are homologous in $\S^-$ as well.
Since $\S^-$ is closed we can now find two subsurfaces $\S_1,\S_2\subset \S^-$ such that
$\partial \S_1=-c\cup d$, and such that (with the orientations induced
from $\S^-$) the following hold: $\S_1\cup \S_2=\S$, $\partial \S_2=c\cup -d$ and $\S_1\cap \S_2=c\cup d$.
Note that possibly one of $\S_1$ or $\S_2$ is disconnected.

\begin{claim}
The surfaces $\ol{\S}_1=\S_1\cup A$ and $\ol{\S}_2=\S_2\cup -A$ are closed, orientable and connected.
 Furthermore, there exists a $j\in \{1,2\}$ such that  $\mbox{genus}(\ol{\S}_j)=\mbox{genus}(\S)$
  and such that $\ol{\S}_j$ is homologous to $ \S$ in $N$.
\end{claim}

\begin{proof}
It is clear that  $\ol{\S}_1$ and $\ol{\S}_2$ are closed, orientable and  connected.
We give $\ol{\S}_k, k=1,2$ the orientation which restricts to the orientation of $\S_k$.
We therefore only have to show the second claim.

Recall that Condition  $(*)$ implies that the inclusion induced maps $H_j(\S^-;\Z)\to H_j(M;\Z), j=0,1$ are isomorphisms. It follows from Lemma \ref{lem:twih2} that  we also have an isomorphism
$H_2(\S^-;\Z)\to H_2(M;\Z)$, in particular  $H_2(M;\Z)$  is generated by $[\S^-]$.
Now note that $\ol{\S}_1$ and $\ol{\S}_2$  represent elements in $H_2(M;\Z)$. We can write $[\ol{\S}_k]=l_k[\S^-], k=1,2$
for some $l_k\in \Z$. Note that $[\ol{\S}_1]+[\ol{\S}_2]=[\S^-]$, i.e.  $l_1+l_2=1$.

Now let $k\in \{1,\dots,r\}$ such that $A_i\subset T_k$.
Recall that we assume that any component of $\S\cap T_k$ represents a nontrivial element in $\pi_1(T_k)$.
By Lemma \ref{lem:nontrivialonsigma} any component of $\S\cap T_k$ therefore also represents a nontrivial element in $\pi_1(\S)$.
In particular $c$ and $d$ do not bound disks on $\S$,
which in turn implies that $\chi(\S_k)\leq 0, k=1,2$. It follows that
\be \label{inequ1} -\chi(\ol{\S}_k)=-\chi((\S^-\sm \S_{3-k})\cup A)=-\chi(\S)+\chi(\S_{3-k})\leq -\chi(\S), \quad k=1,2.\ee
On the other hand, by the linearity of the Thurston norm and the genus minimality of $\S$
we have
\be \label{inequ2} -\chi(\ol{\S}_k)\geq -|l_k|\chi(\S), \quad k=1,2. \ee
 Now recall our assumption that $\chi(\S)=||\phi||_T>0$. It follows that $l_1+l_2=1$ and the inequalities
 (\ref{inequ1}) and (\ref{inequ2}) can only be satisfied if there exists a $j$ with $l_j=1$
and $\chi(\ol{\S}_j)=\chi(\S)$. (Note that necessarily $l_{3-j}=0$ and $\ol{\S}_{3-j}$ is a torus.)
\end{proof}

Note that there exists a proper isotopy of $A\subset M$ to an annulus  $A'\subset M$  such that $\partial A'$ lies entirely in $\S_j$ and such that $A'$ is disjoint from all the other $A_j, j=1,\dots,r$.
We then let $\S_j'\subset \S_j$ be the subsurface of $\S_j$ such that $\partial \S_j'=\partial A'$.
Clearly $\S':=\S_j'\cup -A'$ is isotopic to $\S_j\cup -A$, in particular by the claim $\S'$ is a closed connected surface homologous to $ \S$ in $N$ with $\mbox{genus}(\S')=\mbox{genus}(\S)$ which satisfies all the properties listed in Lemma \ref{lem:minimalsigma}.
On the other hand we evidently have $b_0 (\S'\cap T_j)\leq b_0 (\S)-2$.
Since we did not create any new intersections we in fact have
$\sum_{i=1}^r b_0 (\S'\cap T_i)< \sum_{i=1}^r b_0 (\S \cap T_i)$.
But this contradicts the minimality of $\sum_{i=1}^r b_0 (\S\cap T_i)$ in our choice of the surface $\S$.
We therefore showed that the assumption that $A_i$ does not connect $\S^-$ and $\S^+$ leads to a contradiction.
This concludes the proof of  Lemma \ref{lem:aiconnect}.

\subsection{Preliminaries on the components $M_1,\dots,M_n$}\label{section:guts}

We continue with the notation from the previous sections.
Using Lemma \ref{lem:aiconnect} we can now prove  the following lemma, which in particular implies the first statement of Theorem \ref{thm:miiso}.

\begin{lemma} \label{lem:wi2}
Assume that the inclusion induced maps $\pi_1(\S^\pm)\to \pi_1(M)$ give rise to isomorphisms of the prosolvable completions.
Let $i\in \{1,\dots,n\}$. Then the following hold:
\bn
\item The surfaces $\S_i^\pm$ are connected.
\item For any homomorphism $\a:\pi_1(M)\to S$ to a finite solvable group
the inclusion maps induce isomorphisms
\[ H_j(\S_i^\pm;\Z[S])\to H_j(M_i;\Z[S])\]
for $j=0,1$.
\en
\end{lemma}

\begin{proof}

We first consider statement (1).  Recall that $M_1,\dots,M_n$ are the components of $M$ split along
$A_1,\dots,A_m$. We therefore get the following commutative diagram of Mayer--Vietoris sequences
\[ \xymatrix{
\dots \ar[r]&\bigoplus\limits_{k=1}^m H_j(A_k\cap \S^\pm;\Z) \ar[r]\ar[d] &\bigoplus\limits_{l=1}^n H_j(M_l \cap \S^\pm;\Z)\ar[d]\ar[r]&H_j(\S^\pm;\Z)\ar[d]\ar[r]&\dots \\
\dots \ar[r]&\bigoplus\limits_{k=1}^m H_j(A_k;\Z) \ar[r] & \bigoplus\limits_{l=1}^n H_j(M_l;\Z)\ar[r]&H_j(M;\Z)\ar[r]&\dots. }\]
Note that  the vertical homomorphisms on the left are isomorphisms
since by Lemma \ref{lem:aiconnect} for any $i=1,\dots,m$ the $A_i$ is an annulus which connects $\S^-$ and $\S^+$, i.e. $A_i$ is a product on
$A_i\cap \S^\pm$. Also note that the vertical homomorphisms on the right
are isomorphisms for $j=0,1$ by Proposition \ref{prop:sameh0h1}
and for $j=2$ by  Lemma \ref{lem:twih2}.
We can now appeal to the 5--lemma to deduce that the middle
homomorphisms are isomorphisms as well. But for any $j$ the middle homomorphism is a direct sum of homomorphisms,
it therefore follows in particular that the maps $H_j(\S_i^\pm;\Z)\to H_j(M_i;\Z), j=0,1$ are isomorphisms
for any $i\in \{1,\dots,n\}$.
In particular $b_0(\S_i^\pm)=b_0(M_i)=1$, i.e. the surfaces $\S_i^\pm$ are connected.

We now prove statement (2). Let $\a:\pi_1(M)\to S$ be a homomorphism to a finite solvable group.
Recall that by Lemmas
\ref{lem:im} and \ref{lem:twih2}
we have  that the inclusion induced maps $H_j(\S^\pm;\Z[S])\to H_j(M;\Z[S])$ are   isomorphisms for $j=0,1,2$.
It now follows from the  commutative diagram of Mayer--Vietoris sequences as above, but with $\Z[S]$--coefficients
(cf. \cite{FK06} for details)
that
\[ H_j(\S_i^\pm;\Z[S])\to H_j(M_i;\Z[S]) \]
is an isomorphism for any $i\in \{1,\dots,n\}$ and $j=0,1$.

\end{proof}

The following lemma in particular implies the second statement of Theorem \ref{thm:miiso}.

\begin{lemma} \label{lem:wi1}
For any pair $(i,j)$  such that $M_i\subset N_j$ we have a commutative diagram of injective maps as follows:
\[ \xymatrix{ \pi_1(\S_i^\pm) \ar[d]\ar[r] &\pi_1(M_i) \ar[d]\ar[r]&\pi_1(N_j)\ar[d]\\  \pi_1(\S^\pm) \ar[r] &\pi_1(M)\ar[r]&\pi_1(N).}\]
\end{lemma}

\begin{proof}
First note that since $\S$ is incompressible we know that the two bottom maps are injective.
Furthermore recall that  $N_j$ is a JSJ component of $N$, i.e. a component of the result of cutting
$N$ along incompressible tori, hence $\pi_1(N_j)\to \pi_1(N)$ is injective.

\begin{claim}
For any $k\in \{1,\dots,n\}$ the maps $\pi_1(\S^{\pm}_k)\to \pi_1(\S)$ are injective.
\end{claim}

Let $c$ be a component of
$\S\cap T_l$ for some $l\in \{1,\dots,r\}$. Recall that by our choice of tori $T_1,\dots,T_r$
the curve $c$ represents a nontrivial element in $\pi_1(T_l)$. By Lemma \ref{lem:nontrivialonsigma}
the curve $c$ also  represents a nontrivial element in
$\pi_1(\S)$. In particular none of the components of $\S^\pm \sm \S_k^\pm$ are disks
and therefore the maps $\pi_1(\S_k^\pm)\to \pi_1(\S^\pm)$
are injective. This concludes the proof of the claim.

Now let
$ K= \{ k\in \{1,\dots,n\}\, |\, M_k\subset N_j\}$.
It follows from the claim and the above commutative diagram that for any $k\in K$ the inclusion induced map
$\pi_1(\S_k)\to \pi_1(N_j)$ is injective, i.e. for any $k\in K$ the surface $\S_k\subset N_j$ is incompressible.
Since $M_i$ is a component of cutting $N_j$ along the incompressible surfaces $\S_k^-\subset N_j, k\in K$
we have that $\pi_1(M_i)\to \pi_1(N_j)$ is injective.

By commutativity of the above diagram we now obtain that all other maps are  injective as well.
\end{proof}

\subsection{The conclusion of the proof of Theorem \ref{thm:miiso}}\label{section:hyps}

In this section we will finally  prove the third statement of Theorem \ref{thm:miiso}.
The main ingredient in the proof is the following result.

\begin{proposition}\label{prop:extendhom}
Let $\S$ be a closed surface and $\S'\subset \S$ a connected subsurface such that
 $\pi_1(\S')\to \pi_1(\S)$ is injective.
Let $\a:\pi_1(\S')\to S$ be a homomorphism to a finite solvable group.
Then there  exists a homomorphism  to a finite solvable group $\b:\pi_1(\S)\to T$
and a homomorphism $\pi:T':=\im\{\pi_1(\S')\to T\} \to S$ such that the following diagram commutes:
\[ \xymatrix{ \pi_1(\S')\ar[d]_{\a}\ar[dr]\ar[rr]&& \pi_1(\S)\ar[d]^{\b} \\
S&\ar[l]_{\pi}T'\ar@{->}[r]& T.}\]
\end{proposition}

Put differently, the prosolvable topology on $\pi_1(\S')$ agrees with the topology on $\pi_1(\S')$ induced from the prosolvable topology on $\pi_1(\S)$.

\begin{remark}
Note that  in general
$H_1(\S';\Z)\to H_1(\S;\Z)$ is not injective, even if $\pi_1(\S')\to \pi_1(\S)$ is an injection.
In particular in general
a homomorphism $\pi_1(\S')\to S$ to an abelian group will not extend to a homomorphism from $\pi_1(\S)$ to an abelian group. This shows that in general we can not take $T=S$ or $T$ of the same solvability length as $S$ in the above proposition.
\end{remark}

\begin{proof}
The statement of the proposition is trivial if $\S'=\S$, we will therefore henceforth only consider the
case that $\S'\ne \S$.
Let $\a:\pi_1(\S')\to S$ be a homomorphism to a finite solvable group.
It suffices to show that  there  exists a homomorphism $\b:\pi_1(\S)\to T$ to a finite solvable group
such that $\ker(\b) \cap \pi_1(\S')\subset \ker(\a)$.

Denote by $\S_1,\dots,\S_l$ the components of $\ol{\S\sm \S'}$.
Note that the condition that $\pi_1(\S')\to \pi_1(\S)$ is injective is equivalent to saying that
none of the subsurfaces $\S_1,\dots,\S_l$ is a disk.

It is straightforward to see that for each $j=1,\dots,l$ we can find an annulus $A_j\in \mbox{int}(\S_j)$
such that $(\S' \cup \S_j)\sm A_j$ is still connected.
\begin{figure}[h] \begin{center}
 \includegraphics[scale=0.3]{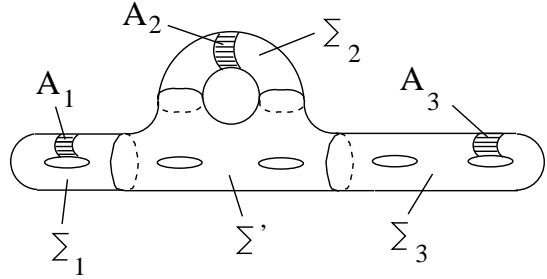} \caption{Surface $\S'\subset \S$ with annuli $A_i\subset \S_i, i=1,2,3$.}
\label{fig:subsurface}
\end{center}
 \end{figure}

Now let $\S''=\ol{\S\sm \cup_{j\in J}A_j}$. Clearly $\S''$ is a connected surface
 with boundary. By assumption $\pi_1(\S')\to \pi_1(\S)$ is injective.
 Since $\pi_1(\S')\to \pi_1(\S)$ factors through $\pi_1(\S'')$ we see that $\S'$ is a subsurface of $\S''$ such that  $\pi_1(\S')\to \pi_1(\S'')$ is injective.
Since $\S''$ is a surface with boundary (contrary to $\S$)  this implies
that $\pi_1(\S')$ is in fact a free factor of $\pi_1(\S'')$, i.e. we have an isomorphism
$\g:\pi_1(\S'')\xrightarrow{\cong} \pi_1(\S')* F$ where $F$ is a free group such that the map $\pi_1(\S'')\xrightarrow{\g}\pi_1(\S')* F\to \pi_1(\S')$
splits the inclusion induced map $\pi_1(\S')\to \pi_1(\S'')$.

We now write $\pi:=\pi_1(\S'')$ and we denote by $\a''$  the projection map
$\pi\to \pi/\pi(S)$ (We refer to  Section \ref{section:defch} for the definition  and the properties
of the characteristic subgroup $\pi(S)$ of $\pi$).
We can extend
$\a:\pi_1(\S')\to S$ to $\pi_1(\S'')\xrightarrow{\g}\pi_1(\S')* F\to \pi_1(\S')\xrightarrow{\a}S$. It follows immediately  that $\ker(\a'')\cap \pi_1(\S') \subset \ker(\a)$.

We will now extend $\a'':\pi_1(\S'')\to \pi/\pi(S)$ to a homomorphism
$\b:\pi_1(\S)\to \Z/n\ltimes \pi/\pi(S)$ where $1\in \Z/n$ acts in an appropriate way
on $\pi/\pi(S)$. In order to do this we will first study the relationship between
$\pi_1(\S'')$ and $\pi_1(\S)$.

Evidently $\S=\S''\, \cup \, \cup_{i=1}^k A_i$.
We pick an orientation for $\S$ and give $A_1,\dots,A_k$ the induced orientations.
We write $\partial A_i=-a_i\cup b_i, i=1,\dots,k$ (see Figure \ref{fig:subsurface2}).
\begin{figure}[h] \begin{center}
 \includegraphics[scale=0.3]{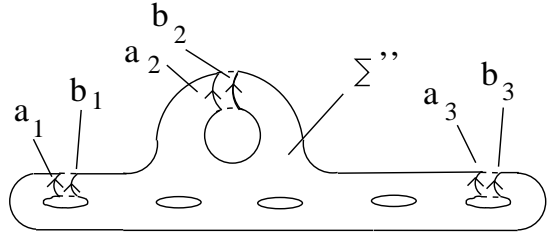} \caption{Surface $\S''\subset \S$ with oriented boundary curves $a_i,b_i$.}
\label{fig:subsurface2}
\end{center}
 \end{figure}
We now pick  a base point for $\S''$. We can find   based curves $c_1,\dots,c_l,d_1,\dots,d_l$
and paths from the base point to the curves $a_1,\dots,a_k, b_1,\dots,b_k$ (and from now on we do not distinguish in the notation
between curves and based curves) such that
\[ \pi = \ll  a_1,\dots,a_k,b_1,\dots,b_k,c_1,\dots,c_l,d_1,\dots,d_l \, |\,
a_1\dots a_k b_k^{-1}\dots b_1^{-1}=[c_l,d_l]\dots [c_1,d_1] \rr. \]
(See Figure \ref{fig:subsurface3} for an illustration.)
\begin{figure}[h] \begin{center}
 \includegraphics[scale=0.3]{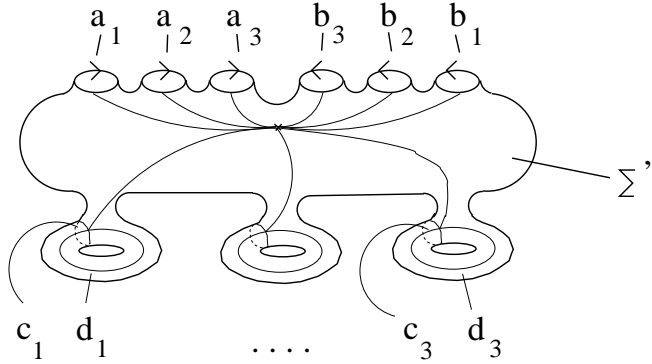} \caption{Surface $\S''\subset \S$ with oriented based curves $a_i,b_i,c_i,d_i$.}
\label{fig:subsurface3}
\end{center}
 \end{figure}
By the van Kampen theorem we then have
\[ \pi_1(\S)=\ll \pi_1(\S''), t_1,\dots,t_k \, |\, t_ia_it_i^{-1}=b_i, i=1,\dots,k \rr.\]

\begin{claim}
There exists an automorphism $\varphi:\pi\to \pi$ such that $\varphi(a_i)=b_i$ and $\varphi(b_i)=a_i$ for any
$i\in \{1,\dots,k\}$.
\end{claim}

Let $\G$ be the free group generated by $a_i,b_i, i=1,\dots,k$ and $c_i,d_i, i=1,\dots,l$ and
consider  the isomorphism $\varphi:\G\to \G$ defined by
$\varphi(a_i)=b_i, \varphi(b_i)=a_i, i=1,\dots,k$ and  $ \varphi(c_i)=d_{l+1-i}, \varphi(d_i)=c_{l+1-i}, i=1,\dots,l$.
In the following we write $w=[c_l,d_l]\dots [c_1,d_1]$
and we write $r=a_1\dots a_k b_k^{-1}\dots b_1^{-1}\cdot  [c_1,d_1]^{-1}\dots [c_l,d_l]^{-1}$ for the relator.
Note that we have a canonical isomorphism $\pi \cong \G/\ll \ll r\rr\rr$.
We calculate
\[ \ba{rclccc} \varphi(r)&=&\varphi\big(a_1\dots a_k b_k^{-1}\dots b_1^{-1}\cdot  [c_1,d_1]^{-1}\dots [c_l,d_l]^{-1}\big)\\[1mm]
&=&b_1\dots b_ka_k^{-1}\dots a_1^{-1}  \cdot [d_l,c_l]^{-1}\dots [d_1,c_1]^{-1}\\[1mm]
&=&b_1\dots b_ka_k^{-1}\dots a_1^{-1}\cdot [c_l,d_l]\dots [c_1,d_1]\\[1mm]
&=&w^{-1}[c_l,d_l]\dots [c_1,d_1] b_1\dots b_ka_k^{-1}\dots a_1^{-1}w\\[1mm]
&=&w^{-1}r^{-1}w.\ea \]
This shows that $\varphi$ restricts to an automorphism of the subgroup of $\G$ normally generated by the relator $r$. In particular $\varphi$ descends to an automorphism of $\pi$.
This concludes the proof of the claim.

Recall that $\pi(S)$ is a characteristic subgroup of $\pi$, hence $\varphi:\pi\to \pi$ descends
to an automorphism $\pi/\pi(S)\to \pi/\pi(S)$ which we again denote by $\varphi$.
Furthermore recall that $\pi/\pi(S)$ is a finite solvable group. Since $\pi/\pi(S)$ is finite  there exists $n>0$ such that
$\varphi^n:\pi/\pi(S)\to \pi/\pi(S)$ acts as the identity.
We can therefore consider the semidirect product $\Z/n \ltimes \pi/\pi(S)$ where $1\in \Z/n$
acts on $\pi/\pi(S)$ via $\varphi$.

It is now straightforward to check that the assignment
\[ \ba{rcl}
g&\mapsto &(0,\a''(g)), \, \,\, g\in \pi_1(\S''),\\
t_i&\mapsto & (1,0)\ea \]
defines a homomorphism
\[ \pi_1(\S)=
\ll \pi_1(\S''), t_1,\dots,t_k \, |\, t_ia_it_i^{-1}=b_i, i=1,\dots,k \rr
\to  \Z/n\ltimes \pi/\pi(S)\]
which we denote by $\b$.
Clearly  $\b:\pi_1(\S)\to \Z/n\ltimes \pi/\pi(S)$
restricts to $\pi_1(\S'')\to \pi/\pi(S)$ and hence
has the required properties.
\end{proof}

We can now  prove the third statement of Theorem \ref{thm:miiso}.

\begin{proof}[Proof of Theorem \ref{thm:miiso} (3)]
In light of Lemma \ref{lem:wi2} (together with  Corollary \ref{cor:group1})  and Lemma \ref{lem:wi1}
it suffices to show the following claim:

\begin{claim}
Let $M$ be a 3--manifold and  $\S\subset \partial M$ such that
$\pi_1(\S)\to \pi_1(M)$ induces an isomorphism of prosolvable completions.
Furthermore let $M'\subset M$ be a submanifold with the following properties:
\bn
\item[(A)] $\S':=\S\cap M'$ is a connected subsurface of $\S'$,
\item[(B)] $\pi_1(\S')\to \pi_1(\S)$ is injective, and
\item[(C)] for any homomorphism $\a:\pi_1(M)\to S$ to a finite solvable group
the inclusion map induces  isomorphisms
\[ H_j(\S';\Z[S])\to H_j(M';\Z[S])\]
for $j=0,1$ and we have
\[ \im\{\pi_1(\S')\to \pi_1(M)\xrightarrow{\a}S\}=
\im\{\pi_1(M')\to\pi_1(M)\xrightarrow{\a}S\}.\]
\en
Then $\pi_1(\S')\to \pi_1(M')$ induces an isomorphism of prosolvable completions.
\end{claim}

By Lemma \ref{lem:pro} we have to show that for any finite solvable group $S$ the map
\[ \i^*:\hom(\pi_1(M'),S) \to \hom(\pi_1(\S'),S) \]
is a bijection.

So let $S$ be a finite solvable group. We first show that
$\i^*:\hom(\pi_1(M'),S) \to \hom(\pi_1(\S'),S)$
is surjective. The various groups and maps in the proof are summarized in the diagram below. Assume we are given
a homomorphism $\a':\pi_1(\S')\to S$.
By (B) and Proposition \ref{prop:extendhom} there exists a homomorphism $\b:\pi_1(\S)\to T$ to a finite solvable group
and a homomorphism $\pi:\im\{\pi_1(\S')\to T\} \to S$ such that $\pi\circ (\b\circ \i)=\a'$.
We write $T'=\im\{\pi_1(\S')\to T\}$ and $\b'=\b\circ \i:\pi_1(\S')\to T'$.

By our assumption that $\pi_1(\S)\to \pi_1(M)$ induce isomorphisms of prosolvable completions
and by Lemma \ref{lem:pro} there exists a homomorphism $\varphi:\pi_1(M)\to T$
such that $\b=\varphi\circ \i$. By (C) we have
\[ \im\{\pi_1(M')\to \pi_1(M)\xrightarrow{\varphi} T\}=\im\{\pi_1(\S')\xrightarrow{\i}\pi_1(M)
\xrightarrow{\varphi} T\}=\im\{\pi_1(\S')\xrightarrow{\b} T\}=T'.\]
Now denote the induced homomorphism $\pi_1(M')\to T'$ by $\varphi'$. Clearly $\varphi'\circ \i=\b'$.
Hence $\a'=\pi\circ \b'=(\pi\circ \varphi')\circ \i$.
This shows that $\i^*:\hom(\pi_1(M'),S) \to \hom(\pi_1(\S'),S)$ is surjective.
The following diagram  summarizes the homomorphisms in the proof of the previous claim:
\[ \xymatrix{ \pi_1(\S')\ar[dr]_{\a'}\ar[drr]^{\b'=\b\circ \i} \ar[dd]_\i\ar[rrrr]^\i &&&& \pi_1(\S) \ar[dl]_{\b}\ar[dd]^\i\\
& S &T' \ar[l]^{\pi\quad} \ar@{(->}[r]&T& \\
 \pi_1(M')\ar[urr]_{\varphi'}\ar[rrrr]^\i &&&& \pi_1(M)\ar[ul]^\varphi.}\]

We now show that $\i^*:\hom(\pi_1(M'),S) \to \hom(\pi_1(\S'),S)$ is injective.
Let $\a_1,\a_2:\pi_1(M')\to S$ be two different homomorphisms.
Let $n$ be the maximal integer such that the homomorphisms
$\pi_1(M')\to S\to S/S^{(n)}$ induced by $\a_1$ and $\a_2$ agree.
We will show that the restriction to $\pi_1(\S')$ of the maps $\pi_1(M')\to S\to S/S^{(n+1)}$
induced by $\a_1$ and $\a_2$  are different. Without loss of generality we can therefore assume that $S=S/S^{(n+1)}$.

We denote by $\psi'$ the homomorphism $\pi_1(M')\to S\to S/S^{(n)}=:G$, induced by $\a_1$ and $\a_2$.

\begin{claim} There  exists a homomorphism $\varphi:\pi_1(M)\to H$ to a finite solvable group
and a homomorphism $\pi:\im\{\pi_1(M')\to \pi_1(M)\to H\} \to G$ such that $\psi'=\pi \circ (\varphi\circ \i)$.
\end{claim}

By (B) and Proposition \ref{prop:extendhom} there exists a homomorphism $\b:\pi_1(\S)\to H$
to a finite solvable group $H$
and a homomorphism $\pi:\im\{\pi_1(\S')\to H\} \to G$ such that $\pi' \circ (\b\circ \i)=\psi' \circ \i$.
By our assumption and by Lemma \ref{lem:pro} there exists a homomorphism $\varphi:\pi_1(M)\to H$
such that $\b=\varphi\circ \i$.
By (C) we have $\im\{\pi_1(\S')\to H\}=\im\{\pi_1(M')\to H\}=:H'$.
It is now clear that $\varphi$ and $\pi$ have the required properties.
This concludes the proof of the claim.

The following diagram summarizes the homomorphisms in the proof of the previous claim:
\[ \xymatrix{ \pi_1(\S')\ar[dr]_{\psi' \circ \i}\ar[drr] \ar[dd]_\i\ar[rrrr]^\i &&&& \pi_1(\S) \ar[dl]_{\b}\ar[dd]^\i\\
& G &H' \ar[l]^{ \pi\quad } \ar@{(->}[r]&H& \\
 \pi_1(M')\ar[ur]^{\psi'}\ar[urr]_{\varphi'=\varphi\circ \i}\ar[rrrr]^\i &&&& \pi_1(M)\ar[ul]^\varphi.}\]
 We now
apply (C)  and
Corollary \ref{cor:metab} to the case $A=\pi_1(\S'), B=\pi_1(M')$ and $\varphi':B\to H'$
to conclude that
the inclusion map induces an isomorphism
\[ \pi_1(\S')/[\ker(\varphi' \circ \i), \ker(\varphi' \circ \i)] \to
 \pi_1(M')/[\ker(\varphi' ), \ker(\varphi'].\]

We now consider the homomorphisms $\a_1,\a_2:\pi_1(M')\to S=S/S^{(n+1)}$.
First note that they factor through  $\pi_1(M')/[\ker(\psi'),\ker(\psi')]$.
Now note that $\ker(\varphi')\subset \ker(\psi')\subset \pi_1(M')$, in particular
we have a  surjection $\pi_1(M')/\ker(\varphi')\to \pi_1(M')/\ker(\psi')$ which
gives rise to a surjection
\[ \pi_1(M')/[\ker(\varphi'),\ker(\varphi')]
\to \pi_1(M')/[\ker(\psi'),\ker(\psi')]. \]
In particular $\a_1,\a_2$ factor through $ \pi_1(M')/[\ker(\varphi'),\ker(\varphi')]$.
We therefore obtain the following commutative diagram
\[
\xymatrix{ \pi_1(\S')\ar[d]\ar[rr]^\i && \pi_1(M')\ar[dl]\ar[ddd]^{\a_1}_{\a_2}\\
 \pi_1(\S')/[\ker(\varphi' \circ \i), \ker(\varphi' \circ \i)] \ar[r]^-{\cong}&
 \pi_1(M')/[\ker(\varphi' ), \ker(\varphi')]\ar[d]&\\
&  \pi_1(M')/[\ker(\psi' ), \ker(\psi')]\ar[dr]&\\
 & &S.}\]
It is now clear that $\a_1 \circ \i$ and $\a_2 \circ \i$ are different.
This concludes the proof that  $\i^*:\hom(\pi_1(M'),S) \to \hom(\pi_1(\S'),S)$ is injective.
As we pointed out before, it now follows from Lemma \ref{lem:pro} that
 $\i:\pi_1(\S')
\to \pi_1(M')$ induces an isomorphism of prosolvable completions.
\end{proof}

\section{The proof of Theorem \ref{mainthm}} \label{section:pgroups}\label{section:mainthm}

We start out with the following two results which allow us to reduce the proof  of Theorem \ref{mainthm}
to the case of closed prime 3--manifolds.

\begin{lemma}\label{lem:irr}\label{lem:prime}
Let $N$ be a 3--manifold with empty or toroidal boundary and let $\phi\in H^1(N;\Z)$ be nontrivial.
If $\Delta_{N,\phi}^\a$ is nonzero for any homomorphism  $\a:\pi_1(N)\to G$ to a finite group $G$,
then $N$ is prime.
\end{lemma}

Note that the main idea for the proof of this lemma can already be found in \cite{McC01}.

\begin{proof}
Let $N$ be a 3--manifold with empty or toroidal boundary which is not prime, i.e. $N=N_1\# N_2$ with $N_1,N_2\ne S^3$.
We have to show that there exists a
 homomorphism $\a:\pi_1(N)\to G$ to a finite group such that $\Delta_{N,\phi}^\a=0$.
 Recall that   by
 Lemma \ref{lem:twiprop2} we have   $\Delta_{N,\phi}^\a=0$ if and only if  $H_1(N;\Q[G]\tpm)$ is not $\qt$--torsion.
 Note that we can write $N=(N_1\sm \intt D^3)\cup_{S^2} (N_2\sm \intt D^3)$
 and that $H_j(N_i\sm \intt D^3;\qt)=H_j(N_i;\qt)$ for $j=0,1$ and $i=1,2$.
 The Mayer--Vietoris sequence corresponding to $N=(N_1\sm \intt D^3)\cup_{S^2} (N_2\sm \intt D^3)$ now gives rise to the following long exact sequence:
 \[ \ba{cccccccccccccccccc} H_1(S^2;\qt)&\to& H_1(N_1;\qt)&\hspace{-0.2cm} \oplus\hspace{-0.2cm}&H_1(N_2;\qt)&\to&H_1(N;\qt)&\to \\
 H_0(S^2;\qt)&\to& H_0(N_1;\qt)&\hspace{-0.2cm} \oplus\hspace{-0.2cm}&H_0(N_2;\qt)&\to&H_0(N;\qt)&\to &0.\ea \]
 A straightforward computation shows that $H_0(S^2;\qt)=\qt$ and $H_1(S^2;\qt)=0$.

First assume that  $b_1(N_i)>0$ for $i=1,2$. Denote by $\phi_i$ the restriction of $\phi:H_1(N;\Q)\to \Q$ to   $H_1(N_i;\Q)$. If $\phi_i$ is
nontrivial for $i=1$ and $i=2$, then
it follows from Lemma \ref{lem:group1} and Lemma \ref{lem:twiprop2} that $H_0(N_i;\qt)$ is $\qt$--torsion
for $i=1,2$. On the other hand we have  $H_0(S^2;\qt)=\qt$. It
follows from the above Mayer--Vietoris  sequence  that $H_1(N;\qt)$ can not be
$\qt$--torsion. On the other hand, if $\phi_i$ is trivial for some $i\in \{1,2\}$, then
$H_1(N_i;\qt)$ is isomorphic to $H_1(N_i;\Q)\otimes \qt$, in particular $H_1(N_i;\qt)$ is not $\qt$--torsion, and using that
$H_1(S^2;\qt)=0$ it follows again from the above Mayer--Vietoris sequence that $H_1(N;\qt)$ is not
$\qt$--torsion.

Now assume that  either $b_1(N_1)=0$ or $b_1(N_2)=0$. Without loss of generality we can assume that  $b_1(N_2)=0$. Since $b_1(N)=b_1(N_1)+b_1(N_2)$ we have  $b_1(N_1)>0$. By the Geometrization
Conjecture $\pi_1(N_2)$ is nontrivial and  residually finite (cf.
\cite{Th82} and \cite{He87}), in particular  there exists an epimorphism $\a:\pi_1(N_2)\to G$ onto a
nontrivial  finite group $G$. Denote the homomorphism
$\pi_1(N)=\pi_1(N_1)*\pi_1(N_2)\to \pi_1(N_2)\to G$ by $\a$ as well. Then by Lemma \ref{lem:twiprop}  we have
\[ \Delta_{N,\phi}^\a=\Delta_{N_G,\phi_G}\]
where $p:N_G\to N$ is the cover of $N$ corresponding to $\a$ and $\phi_G=p^*(\phi)$. But the prime
decomposition of $N_G$ has $|G|$ copies of $N_1$. By the  argument above we now have
that $\Delta_{N_G,\phi_G}=0$, which  implies that  $\Delta_{N,\phi}^\a=\Delta_{N_G,\phi_G}=0$.
\end{proof}

\begin{lemma}\label{lem:double}\label{lem:closed}
Let $N$ be an irreducible $3$--manifold with non--empty toroidal boundary  and let $\phi \in H^{1}(N;\Z)$ be nontrivial.
Let $W=N\cup_{\partial N} N$ be the double of $N$ along the boundary of $N$.
Let $\Phi=p^*(\phi) \in H^{1}(W;\Z)$ where $p:W\to N$ denotes the folding map.
Then the following hold:
\bn
\item $(W,\Phi)$ fibers over $S^1$ if and only if $(N,\phi)$ fibers over $S^1$,
\item if $(N,\phi)$ satisfies Condition  $(*)$, then $(W,\Phi)$ satisfies Condition  $(*)$.
\en
\end{lemma}

In the proof of Lemma \ref{lem:closed} we will make use of the following well--known lemma.
We refer to \cite[Theorem~4.2]{EN85} and \cite{Ro74}  for the first statement, and to \cite[p.~33]{EN85} for the second statement.

\begin{lemma} \label{lem:split}
Let $Y$ be a closed 3--manifold. Let $T\subset Y$ be a union
 of incompressible tori
such that $T$ separates $Y$ into two connected components $Y_1$ and $Y_2$.
Let $\psi\in H^1(Y;\Z)$. Then the following hold:
\bn
\item If $||\phi||_{T,Y}>0$, then $(Y,\psi)$ fibers over $S^1$ if and only if $(Y_1,\psi|_{Y_1})$ and $(Y_2,\psi|_{Y_2})$ fiber over $S^1$,
\item $||\psi||_{T,Y}=||\psi|_{Y_1}||_{T,Y_1}+||\psi|_{Y_2}||_{T,Y_2}$.
\en
\end{lemma}

\begin{proof}[Proof of Lemma \ref{lem:closed}]
First note that an irreducible 3--manifold with boundary a union of tori  has compressible boundary if and only if it is the solid torus.
Since the lemma holds trivially in the case that $N=S^1\times D^2$ we will from now on assume that $N$
has incompressible boundary. This implies in particular that $||\phi||_T>0$.
The first statement is now an immediate consequence of Lemma \ref{lem:split}
and the observation that $\Phi|_N=\phi$.

Now assume that $(N,\phi)$ satisfies Condition  ($*$). In the following we write $N_i=N, i=1,2$
and we think of $W$ as $W=N_1\cup_{\partial N_1=\partial N_2}N_2$.
Let  $\a:\pi_1(N)\to G$ be a homomorphism to a finite group $G$. We write $n=|G|$,  $V=\Z[G]$ and we slightly abuse notation by denoting by $\a$ the representation $\pi_1(W)\to \aut(V)$ given by left multiplication.
We have to show that  $\Delta_{W,\Phi}^{\a} \in \zt$ is monic
and that \[ \deg(\Delta_{W,\Phi}^{\a})-\deg(\Delta_{W,\Phi,0}^{\a})-\deg(\Delta_{W,\Phi,2}^{\a})= n \, \|\Phi\|_{T}\]
(here we used Lemma  \ref{lem:twiprop} to rephrase the last condition).
For any submanifold $X\subset W$ we denote the restriction of $\Phi$ and $\a$ to $\pi_1(X)$ by $\Phi$ and $\a$ as well. Evidently the restriction of $\Phi$ to $N=N_i, i=1,2$ just agrees with $\phi$.

In order to prove the claims on $\Delta_{W,\Phi}^{\a}$ we will in the following express   $\Delta_{W,\Phi}^{\a}$  in terms of
$\Delta_{N_i,\phi_i}^\a$, $i=1,2$.
The following statement combines  the assumption that $(N,\phi)$ satisfies Condition  $(*)$ with Lemmas
\ref{lem:twiprop} and \ref{lem:twiprop2}.

\begin{facta}
For $i=1,2$ we have
\[ \deg(\Delta_{N_i,\Phi}^{\a})-\deg(\Delta_{N_i,\Phi,0}^{\a})-\deg(\Delta_{N_i,\Phi,2}^{\a})= n \, \|\Phi\|_{T,N_i}.\]
Furthermore  for all $j$ we have that $\Delta_{N_i,\Phi,j}^{\a}$ is monic.
\end{facta}

We now turn to the twisted Alexander polynomials of the boundary tori of $\partial N$.
The following is an immediate consequence of Theorem \ref{thm:fv08torus}.

\begin{factb}
If $\Delta_{N,\phi}\ne 0$ (in particular if $(N,\phi)$ satisfies Condition  $(*)$), then for any boundary component $T\subset \partial N$
the restriction of $\phi$ (and hence of $\Phi$) to
$\pi_1(T)$  is nontrivial.
\end{factb}

This fact and a straightforward computation now gives us the following fact
(cf. e.g. \cite{KL99}).

\begin{factc}
Let $T\subset \partial N$ be any boundary component.
Then
\bn
\item $\Delta_{T,\Phi,i}^\a$ is monic for any $i$,
\item $H_i(T;V\tpm)=0$ for $i\geq 2$, in particular $\Delta_{T,\Phi,i}^\a=1$ for $i\geq 2$,
\item $\Delta_{T,\Phi,0}^\a= \Delta_{T,\Phi,1}^\a$.
\en
\end{factc}

We now consider the following Mayer--Vietoris sequence:
\[
\ba{ccccccccccccccccc}
0&\to& H_2(N_1;V\tpm)\oplus H_2(N_2;V\tpm)&\to & H_2(W;V\tpm)&\to &\\
H_1(\partial N;V\tpm)&\to& H_1(N_1;V\tpm)\oplus H_1(N_2;V\tpm)&\to & H_1(W;V\tpm)&\to &\\
H_0(\partial N;V\tpm)&\to& H_0(N_1;V\tpm)\oplus H_0(N_2;V\tpm)&\to & H_0(W;V\tpm)&\to &0.\ea \]
Recall that we assume that $(N,\phi)$ (and hence $(N_i,\phi), i=1,2$) satisfy Condition ($*$).
By  Lemmas \ref{lem:twiprop} and \ref{lem:twiprop2} and Facts 1 and 3 it follows that
all homology modules in the above long exact sequence but possibly $H_1(W;V\tpm)$ and $H_2(W;V\tpm)$  are $\zt$--torsion.
But then evidently $H_1(W;V\tpm)$ and $H_2(W;V\tpm)$ also have to be  $\zt$--torsion.
Furthermore it follows from Fact 3, \cite[Theorem~3.4]{Tu01} and \cite[Theorem~4.7]{Tu01}
that
\be \label{equ:deltas} \frac{ \Delta_{W,\Phi,1}^\a}{ \Delta_{W,\Phi,0}^\a \Delta_{W,\Phi,2}^\a}=
\frac{ \Delta_{N_1,\Phi,1}^\a}{ \Delta_{N_1,\Phi,0}^\a\Delta_{N_1,\Phi,2}^\a}\cdot \frac{ \Delta_{N_2,\Phi,1}^\a}{ \Delta_{N_2,\Phi,0}^\a\Delta_{N_2,\Phi,0}^\a}.
\ee
Note that $ \Delta_{W,\Phi,0}^\a$ and $\Delta_{W,\Phi,2}^\a$ are monic by Lemma \ref{lem:twiprop},
it now follows from Fact 1 and Equality (\ref{equ:deltas}) that $ \Delta_{W,\Phi,1}^\a$ is monic as desired.

Finally we can appeal to Lemma \ref{lem:split} to conclude that $||\Phi||_{T,W}=||\Phi||_{T,N_1}+||\Phi||_{T,N_2}$.
It therefore follows from Fact 1 and Equation (\ref{equ:deltas}) that
\[ \deg(\Delta_{W,\Phi}^{\a})-\deg(\Delta_{W,\Phi,0}^{\a})-\deg(\Delta_{W,\Phi,2}^{\a})= n \, \|\Phi\|_{T,W}\]
as required.
\end{proof}

Let $N$ be a $3$--manifold with empty or toroidal boundary. We write $\pi=\pi_1(N)$.
Let $\tipi\subset \pi$ be a finite index subgroup  and $\phi \in H^{1}(N;\Z)$ nontrivial. We now say that the pair $(\tipi,\phi)$ has Property (M) if
the twisted Alexander polynomial $\Delta_{N,\phi}^{\pi/\tipi}\in \zt$ is monic
and if \[ \deg(\Delta_{N,\phi}^{\pi/\tipi})= [\pi:\tipi] \, \|\phi\|_{T} + (1+b_3(N)) \div  \phi_{\tipi}\] holds.

The first statement of the following lemma is well--known, the second one can be easily verified and the third is an immediate consequence of the second statement.

\begin{lemma}\label{lem:kphi}
Let $N$ be a $3$--manifold with empty or toroidal boundary and let $\phi \in H^{1}(N;\Z)$ be nontrivial.
Let $k\ne 0\in \Z$. Then the following hold:
\bn
\item $(N,\phi)$ fibers over $S^1$ if and only if $(N,k\phi)$ fibers over $S^1$,
\item  Let $\tipi\subset \pi$ be a finite index subgroup. Then $(\tipi,\phi)$ has Property (M) if and only if
$(\tipi,k\phi)$ has Property (M),
\item $(N,\phi)$ satisfies Condition  $(*)$ if and only if $(N,k\phi)$ satisfies Condition  $(*)$.
\en
\end{lemma}

We will also need the following lemma.

\begin{lemma}\label{lem:nonnormal}
Let $N$ be a $3$--manifold with empty or toroidal boundary and let $\phi \in H^{1}(N;\Z)$ be non--trivial.
Suppose that all finite index \emph{normal} subgroups of $\pi_1(N)$ have Property (M), then in fact all finite index subgroups of $\pi_1(N)$ have Property (M).
\end{lemma}

\begin{proof}
We write $\pi:=\pi_1(N)$.
Let $\phi \in H^{1}(N;\Z)$ be non--trivial. By Lemma \ref{lem:kphi} (2) we can without loss of generality assume that $\phi$ is primitive.
Let $\tipi\subset \pi$ be a finite index subgroup.
We denote by $\hatpi\subset \pi$ the core of $\tipi$, i.e. $\hatpi=\cap_{g\in \pi} g\tipi g^{-1}$.
Note that $\hatpi$ is normal in $\pi$ and contained in $\tipi$.

By Proposition \ref{prop:mcm} the class $\phi$ is dual to a connected Thurston norm minimizing surface $\S$.
We write $A=\pi_1(\S)$ and $B=\pi_1(N\sm \nu \S)$ as before.

We write $\hat{B}:=B\cap \hatpi$ and $\hat{A}^\pm:=(\i_\pm)^{-1}(\hat{B})$.
We now pick representatives $g_1,\dots,g_m$ for the equivalence classes of $B\backslash \pi/\ti{\pi}$.
For $i=1,\dots,m$ we write $\ti{B}_i:=B\cap g_i\tipi g_i^{-1}$ and $\ti{A}^\pm_i:=(\i_\pm)^{-1}(\ti{B}_i)$.

Since $\hatpi\subset \pi$ is normal and since we assume that normal finite index subgroups have Property (M)
we can now apply Proposition  \ref{prop:sameh0h1} and Lemma \ref{lem:group2} to conclude that
\[ \i_\pm: H_j(A;\Z[B/\hat{B}])\to H_j(B;\Z[B/\hat{B}]) \]
are isomorphisms for $j=0,1$.
It now follows from
 Corollary \ref{cor:metabany}
 that the maps
\[ \i_\pm:A/  \hat{A}^\pm \to B/\hat{B}  \mbox{ and } \i_\pm:A/ [ \hat{A}^\pm,\hat{A}^\pm  ]\to B/[\hat{B},\hat{B}]\]
are isomorphisms.
Recall that $\hatpi$ is normal in $\pi$, it follows that  $\hat{B}\subset B$ is normal and for any $i$ we have $\hat{B}=B\cap g_i\hatpi g_i^{-1}\subset B\cap g_i\tipi g_i^{-1}=\ti{B}_i$. We now deduce from
 Lemma \ref{lem:alsoiso} that
\[  \i_\pm:A/  \ti{A}^\pm_i \to B/\ti{B}_i  \mbox{ and }  \i_\pm:A/ [ \ti{A}^\pm_i,\ti{A}^\pm_i  ]\to B/[\ti{B}_i,\ti{B}_i]\]
are  bijections for $i=1,\dots,m$.
It now follows from Lemma \ref{lem:metabany}
that the maps
\[ \i_\pm: H_j(A;\zpp)\to  H_j(B;\zpp) \]
are isomorphisms. It now follows from  Proposition  \ref{prop:sameh0h1} that
$\tipi$ also  has Property (M).
\end{proof}



We will now use the previous lemma to prove the following lemma.

\begin{lemma} \label{lem:pullback}
Let $N$ be a $3$--manifold with empty or toroidal boundary and let $\phi \in H^{1}(N;\Z)$ be non--trivial.
Let $p:N'\to N$ be a finite cover. We write $\phi'=p^*(\phi) \in H^1(N';\Z)$.
Then the following hold:
\bn
\item $\phi'$ is nontrivial,
\item $(N,\phi)$ fibers over $S^1$ if and only if $(N',\phi')$ fibers over $S^1$,
\item if $(N,\phi)$ satisfies Condition  $(*)$, then $(N',\phi')$ satisfies Condition  $(*)$.
\en
\end{lemma}

\begin{proof}
The first statement is well--known. The second statement  is a consequence of \cite[Theorem~10.5]{He76}.
We now turn to the third statement. Assume that $(N,\phi)$ satisfies Condition  ($*$).

Let $\tipi$ be a normal finite index subgroup of $ \pi'=\pi_1(N')$. We have to show that $(\tipi,\phi')$ has Property (M).
Note that $\tipi$ viewed as a subgroup of $\pi=\pi_1(N)$ is not necessarily normal.
It nonetheless follows from the assumption that $(N,\phi)$ satisfies Condition  ($*$)
and from Lemma \ref{lem:nonnormal} that
the twisted Alexander polynomial $\Delta_{N,\phi}^{\pi/\tipi}\in \zt$ is monic
and that \[ \deg(\Delta_{N,\phi}^{\pi/\tipi})= [\pi:\tipi] \, \|\phi\|_{T} + (1+b_3(N)) \div  \phi_{\tipi}\] holds.
It now
follows easily from Lemma \ref{lem:twiprop},
 $b_3(N)=b_3(N')$, and
 the multiplicative property of the Thurston norm under finite covers (cf. \cite[Corollary~6.13]{Ga83})
 that the twisted Alexander polynomial $\Delta_{N',\phi'}^{\pi'/\tipi}\in \zt$ is monic
and that the following equality holds:
\[ \deg(\Delta_{N',\phi'}^{\pi'/\tipi})= [\pi':\tipi] \, \|\phi'\|_{T} + (1+b_3(N')) \div  \phi'_{\tipi}.\]
In particular $(\tipi,\phi')$ has Property (M).
\end{proof}

We are now finally  in a position to prove Theorem \ref{mainthm}.

\begin{proof}[Proof of Theorem \ref{mainthm}]
First note that the combination of Theorem   \ref{thm:jsjresp}
and Lemmas \ref{lem:prime}, \ref{lem:closed}, \ref{lem:pullback} and \ref{lem:kphi}
 shows that it suffices to show the following claim:

 \begin{claim}
Assume we are given  a pair $(N,\phi)$
 where
 \bn
 \item   $N$ is a closed irreducible 3--manifold such that the fundamental group
 of each JSJ component is residually $p$, and
 \item  $\phi$  is  primitive.
 \en
If $(N,\phi)$ satisfies Condition  $(*)$, then $(N,\phi)$ fibers over $S^1$.
\end{claim}

Let $(N,\phi)$ be a pair as in the claim which satisfies Condition  $(*)$.
If $||\phi||_T=0$, then it follows from \cite[Proposition~4.6]{FV08b} that $(N,\phi)$ fibers over
$S^1$.

We can and will therefore henceforth assume that $||\phi||_T>0$.
We denote the tori of the JSJ decomposition of $N$ by $T_1,\dots,T_r$.
We pick a connected Thurston norm minimizing surface $\S$ dual to $\phi$ and
a tubular neighborhood $\nu\S=\S\times [-1,1]\subset N$
as in Section \ref{section:deti}.
In particular  we  can and will throughout assume that $\S\times t$ and the tori $T_1,\dots,T_r$ are in general position for any $t\in [-1,1]$
and that for any $i\in \{1,\dots,r\}$ any component of $\S\cap T_i$ represents a nontrivial element in $\pi_1(T_i)$. Furthermore as in Section \ref{section:jsj}
we assume that our choice of $\S$ minimizes the number $\sum_{i=1}^r b_0 (\S\cap T_i)$.

Let $A_1,\dots,A_m$ be the components of the intersection of the tori $T_1,\dots,T_r$ with $M:=N\sm \S\times (-1,1)$.
Furthermore let $M_1,\dots,M_n$ be the components of $M$ cut along $A_1\cup \dots \cup A_m$.
Recall that any $M_i$ is a submanifold of a JSJ component of $N$.

For $i=1,\dots,m$ write $C_i=A_i\cap \S^-$.
It follows from Lemma \ref{lem:aiconnect} that for $i=1,\dots,m$ the surface $A_i$ is an annulus which is a product on $C_i$, i.e. $C_i$ consists of one component and
$\pi_1(C_i)\to  \pi_1(A_i)$ is an isomorphism.

In order to show that $M$ is a product on $\S^-$ it suffices to show
that $\pi_1(\S_i^-)\to \pi_1(M_i)$ is an isomorphism  for any  $i\in \{1,\dots,n\}$.
So let $i\in \{1,\dots,n\}$.
Since $(N,\phi)$ satisfies Condition $(*)$ it follows from Proposition \ref{thm:finitesolvable}  that the maps
$\pi_1(\S^\pm)\to \pi_1(M)$ induce an isomorphism of prosolvable completions.
By Theorem \ref{thm:miiso} (1)  the surfaces $\S_i^\pm$ are connected,
and by Theorem \ref{thm:miiso} (3) the inclusion induced maps $\pi_1(\S_i^\pm)\to \pi_1(M_i), i=1,\dots,n$ also induce  isomorphisms
of prosolvable completions.
By Theorem \ref{thm:miiso} (2) we have that the group $\pi_1(M_i)$ is a subgroup of the fundamental group of a JSJ component of
$N$. By our assumption this implies that $\pi_1(M_i)$ is residually $p$, in particular residually finite solvable.

In the following we view $M_i$ as a sutured manifold with
sutures given by
$\gamma_i=\partial N\cap M_i$.
We  can pick orientations such that  $R_-(\gamma_i)=\S_i^-$ and
$R_+(\gamma_i)=\S_i^+$. Since $\S\subset N$ is Thurston norm minimizing it follows that $(M_i,\gamma_i)$ is a taut sutured
manifold.
We can therefore now apply Theorem \ref{thm:no} to conclude that $(M_i,\gamma_i)$ is a product sutured manifold,
i.e. $\pi_1(\S_i^-)\to \pi_1(M_i)$ is an isomorphism.
\end{proof}

\end{document}